\documentclass{amsart}
\usepackage{setspace}
\usepackage{a4}
\usepackage{amssymb,amsmath,amsthm,latexsym}
\usepackage{amsfonts}
\usepackage{amsfonts}
\usepackage{graphicx}
\usepackage{textcomp}
\usepackage{cite}
\usepackage{enumerate}
\usepackage[mathscr]{euscript}
\usepackage{mathtools}
\newtheorem{theorem}{Theorem}[section]

\newtheorem{conjecture}[theorem]{Conjecture}

\newtheorem{definition}[theorem]{Definition}
\newtheorem{example}[theorem]{Example}

\newtheorem{problem}[theorem]{Problem}

\newtheorem{remark}[theorem]{Remark}
\title{This is the title}
\usepackage{amssymb}
\usepackage{amssymb}
\usepackage{amssymb}
\usepackage{amssymb}
\usepackage{amsmath}
\usepackage{tikz}
\usepackage{hyperref}
\usepackage{enumerate}
\usepackage{mathtools}
\usepackage{amsmath}
\usepackage{tikz}
\usepackage{amssymb}
\usepackage{amsmath}
\usepackage{tikz}
\usepackage{hyperref}
\usepackage{fancyhdr}

\begin{document}
\begin{center}
{\bf{FEICHTINGER CONJECTURES, $R_\varepsilon$-CONJECTURE AND WEAVER'S CONJECTURES  FOR  BANACH SPACES}}\\
\textbf{K. MAHESH KRISHNA} \\
Post Doctoral Fellow\\
Statistics and Mathematics Unit\\
Indian Statistical Institute, Bangalore Centre\\
Karnataka 560 059 India\\
Email: kmaheshak@gmail.com \\

Date: \today
\end{center}

\hrule
\vspace{0.5cm}
\textbf{Abstract}: Motivated from  two decades old famous Feichtinger conjectures for frames, $R_\varepsilon$-conjecture and Weaver's conjecture  for Hilbert spaces (and their solution by Marcus, Spielman, and Srivastava), we formulate Feichtinger conjectures for p-approximate Schauder frames, $R_\varepsilon$-conjecture,  Weaver's conjectures and Akemann-Weaver conjectures   for Banach  spaces. We also formulate  conjectures on p-approximate Schauder frames based on the results  of Casazza for frames for Hilbert spaces. We state conjectures and problems for p-approximate Schauder frames based on fundamental inequality for frames for Hilbert spaces and scaling problem for Hilbert space frames. Based on Kothe-Lorch theorem for Riesz bases for Hilbert spaces, we formulate a problem for p-approximate Riesz bases for Banach spaces. We formulate dynamical sampling problem for p-approximate Schauder frames for Banach spaces. We ask phase retrieval problem and norm retrieval problem for p-approximate Schauder frames for Banach spaces. We also formulate discretization problem for continuous p-approximate Schauder frames.


\hrule
\tableofcontents
\section{Introduction}
Let    $\mathcal{H}$ be a Hilbert space. Recall that a sequence  $\{\tau_n\}_n$ in  $\mathcal{H}$ is said to be a \textbf{frame} for $\mathcal{H}$  \cite{OLEBOOK, DUFFIN, HEILFRAME} if there exist $a,b>0$ such that 
\begin{align}\label{FE}
	a\|h\|^2 \leq \sum_{n=1}^\infty |\langle h, \tau_n\rangle|^2\leq b\|h\|^2, \quad \forall h \in \mathcal{H}.
\end{align}
Constants $a$ and $b$ are called as \textbf{lower frame bound} and \textbf{upper frame bound}, respectively. If 
\begin{align*}
	\sum_{n=1}^\infty |\langle h, \tau_n\rangle|^2=\|h\|^2, \quad \forall h \in \mathcal{H},
\end{align*}
then we say $\{\tau_n\}_n$ is a Parseval frame for   $\mathcal{H}$. If we do not demand the first inequality in  Inequality (\ref{FE}), then we say that  $\{\tau_n\}_n$ is a \textbf{Bessel sequence} for   $\mathcal{H}$. Constant  $b$ is called as \textbf{Bessel bound}.
Notion which is stronger than that of frames is that of Riesz bases defined as follows.
\begin{definition} \label{RIESZ}\cite{BARI, OLEBOOK, BARI2}
	\begin{enumerate}[\upshape(i)]
		\item 	A sequence $\{\tau_n\}_n$ in  $\mathcal{H}$ is said to be a \textbf{Riesz basis} for 
		$\mathcal{H}$ if there exists   a bounded invertible linear operator $T: \mathcal{H}\to \mathcal{H}$ such that 
		\begin{align*}
			T\omega_n=\tau_n, \quad \forall n \in \mathbb{N}, 
		\end{align*}
	where $\{\omega_n\}_n$ is an orthonormal basis for    $\mathcal{H}$.
		\item 	A sequence $\{\tau_n\}_n$ in  $\mathcal{H}$ is said to be a \textbf{Riesz sequence}  for 
		$\mathcal{H}$ if it is a Riesz basis for $\overline{span}\{\tau_n\}_n$. Constants $a,b>0$ satisfying  for all $m\in \mathbb{N}$, 
		\begin{align*}
			a\left(\sum_{n=1}^{m}|c_n|^2 \right)^\frac{1}{2}\leq \left\|\sum_{n=1}^{m}c_n\tau_n\right\|\leq b\left(\sum_{n=1}^{m}|c_n|^2 \right)^\frac{1}{2}\quad , \forall c_1, \dots, c_m \in\mathbb{R} \text{ or } \mathbb{C}
		\end{align*}
		are called as \textbf{lower Riesz bound} and \textbf{upper Riesz  bound}, respectively.
	\end{enumerate}
\end{definition}
Towards the end of $20^{th}$ century, works of Casazza, Christensen, and Lindner \cite{CASAZZACHRISTENSEN1996, CHRISTENSEN1996, CHRISTENSENLINDNER, CASAZZACHRISTENSEN1998, CASAZZALOCAL}      showed that there are frames which do not contain Schauder basis, in particular Riesz basis. On the other hand, work on Gabor frames led Feichtinger to formulate the following conjectures in the beginning years of $21^{th}$ century (see \cite{CHRISTENSENSEVEN} for the history).
\begin{conjecture}\cite{CASAZZACHRISTENSENVER}\textbf{(Feichtinger conjecture for frames)}\label{FCON}
\textbf{Let  $\{\tau_n\}_n$ be  a frame for   $\mathcal{H}$ such that 
	\begin{align*}
		0<\inf_{n\in \mathbb{N}}\|\tau_n\|.
	\end{align*}
	Then $\{\tau_n\}_n$ can be partitioned into a finite union of Riesz sequences. Moreover, what is the number of partitions required?}
\end{conjecture}
\begin{conjecture}\cite{CASAZZACHRISTENSENVER}\label{FCB}\textbf{(Feichtinger conjecture for Bessel sequences)}
	\textbf{Let  $\{\tau_n\}_n$ be  a Bessel sequence  for   $\mathcal{H}$ such that 
		\begin{align*}
			0<\inf_{n\in \mathbb{N}}\|\tau_n\|.
		\end{align*}
		Then $\{\tau_n\}_n$ can be partitioned into a finite union of Riesz sequences. Moreover, what is the number of partitions required?}
\end{conjecture}
\begin{conjecture}\cite{CASAZZACHRISTENSENVER}\textbf{(Finite dimensional Feichtinger conjecture for frames)}\label{FDFCF}
\textbf{Let 	 $\mathcal{H}$ be a $d$-dimensional Hilbert space. For every    real $b,c>0$, there exist a natural number $M(b,c)$,  a real $a(b,c)>0$ so that whenever $\{\tau_j\}_{j=1}^n$ is a frame   for   $\mathcal{H}$ with upper frame bound $b$ and $\|\tau_j\|\geq c$, $\forall 1\leq j \leq n$, then the set $\{1,2, \dots, n\}$ can be partitioned into sets $I_1,I_2,\dots,  I_{M(b,c)}$ so that for each $1\leq k \leq M(b,c)$, $\{\tau_j\}_{j\in I_k}$ is a Riesz sequence with lower Riesz bound $a(b,c)$ and upper Riesz bound $b$. Moreover, what is the number of partitions required?}
\end{conjecture}
\begin{conjecture}\cite{CASAZZACHRISTENSENVER}\textbf{(Finite dimensional Feichtinger conjecture for Bessel sequences)}\label{FDFCB}
\textbf{Let 	 $\mathcal{H}$ be a $d$-dimensional Hilbert space. For every $b>0$, there exists a natural number $M(b)$ and a real  $a(b)>0$ so that for every (Bessel) sequence $\{\tau_j\}_{j=1}^n$ for $\mathcal{H}$  with Bessel bound $b$ and $\|\tau_j\|=1$, $\forall 1\leq j \leq n$, can be written as a union of $M(b)$ Riesz  sequences each with lower Riesz bound $a(b)$. Moreover, what is the number of partitions required?}
\end{conjecture}
After the formulation of Conjectures  \ref{FCON}, \ref{FCB}, \ref{FDFCF} and  \ref{FDFCB}, several of their  equivalent conjectures  were  found and several particular cases  of Conjectures \ref{FCON}, \ref{FCB},  \ref{FDFCF} and  \ref{FDFCB}, have been solved \cite{BALANCASAZZAHEILLANDAU, CASAZZAKUTYNIOKSPEEGLE, PAULSEN, WEBER, LAWTON, BOWNIKCASAZZA2019, LATA, CASAZZACHRISTENSENVER, GAVRUTA, BARANOVDYAKONOV, BOWNIKSPEEGLE,  GROCHENIG}. Conjectures \ref{FCON}, \ref{FCB},  \ref{FDFCF} and  \ref{FDFCB} received a lot of importance after establishing their  equivalence  with Kadison-Singer conjectures  \cite{KADISONSINGER, CASAZZABOOK, CASAZZATREMAINSMALL, CASAZZAEDIDIN, CASAZZATREMAINLARGE, CASAZZAEDIDINEQUI}. Finally, the Feichtinger conjectures have been solved fully  resolving Weaver's conjecture by Marcus, Spielman, and Srivastava in 2013 using an entirely new method called ``mixed characteristic polynomials" by them  \cite{MARCUSSPIELMANSRIVASTAVA2, BOWNIK, CASAZZACONSEQUENCES, MARCUSSPIELMANSRIVASTAVA, TIMOTINTH, MARCUSCD}. Later, Conjecture \ref{FCON} has been set for tight p-frames for $\ell^p(\mathbb{N})$ in \cite{LIULIU} (see \cite{CHRISTENSENSTOEVA2003, STOEVA2006, STOEVA2005, CASAZZACHRISTESENSCH, CASAZZACHRISTENSENSTOEVA, ALDROUBI} for p-frames)    and solved by using Clarkson's inequalities \cite{CLARKSON} in the same paper. In this paper, we formulate  Conjectures \ref{FCON}, \ref{FCB},  \ref{FDFCF} and  \ref{FDFCB} \ref{FCON} for p-approximate Schauder frames for Banach spaces. We also formulate $R_\varepsilon$-conjecture,   Weaver's conjectures and Akemann-Weaver conjectures  for Banach spaces. Three  conjectures based on results of Casazza are formulated. We state scaling problem, fundamental inequality problem, Kothe-Lorch problem  for p-approximate Schauder frames for Banach spaces. We formulate dynamical sampling problem and phase and norm retrieval problems for p-approximate Schauder frames. We also state some other problems including discretization problem for continuous p-approximate Schauder frames for Banach spaces.

\section{Feichtinger conjectures,  $R_\varepsilon$-conjecture, Weaver's conjectures,  scaling problem, dynamical sampling problems,  phase and norm retrieval problems and  discretization problem for p-approximate Schauder frames for Banach spaces}
 Let  $p\in[1, \infty)$. Let $\{e_n\}_n$ be the standard Schauder basis for $\ell^p(\mathbb{N})$  and $\{\zeta_n\}_n$ be the co ordinate functionals associated with $\{e_n\}_n$. Throughout the paper, $\mathcal{X}$ denotes a Banach space. In the process of characterization of approximate Schauder frames and its duals (like that of Holub \cite{HOLUB}  and Li \cite{LI}) for Banach spaces \cite{FREEMANODELL, THOMAS, CASAZZA, CASAZZAHANLARSON}, the notion of p-approximate Schauder frames has been introduced by Krishna and Johnson whose detailed study has been done in the thesis of the  author \cite{MAHESHTHESIS} which reads as follows. 
\begin{definition}\cite{MAHESHJOHNSON, MAHESHJOHNSON3}
Let  $p\in[1, \infty)$.	Let $\{\tau_n\}_n$ be a sequence in a Banach space  $\mathcal{X}$ and 	$\{f_n\}_n$ be a sequence in  $\mathcal{X}^*$ (dual of  $\mathcal{X}$). The pair $ (\{f_n \}_{n}, \{\tau_n \}_{n}) $ is said to be a \textbf{p-approximate Schauder frame (p-ASF)} for $\mathcal{X}$ if the following conditions are satisfied.
	\begin{enumerate}[\upshape(i)]
		\item The map (\textbf{frame operator})
		\begin{align*}
			S_{f, \tau}:\mathcal{X}\ni x \mapsto S_{f, \tau}x\coloneqq \sum_{n=1}^\infty
			f_n(x)\tau_n \in
			\mathcal{X}
		\end{align*}
		is a well-defined bounded linear, invertible operator.	
		\item The map  (\textbf{analysis operator})
		\begin{align*}
\theta_f: \mathcal{X}\ni x \mapsto \theta_f x\coloneqq \{f_n(x)\}_n \in \ell^p(\mathbb{N})
		\end{align*}
		is a well-defined bounded linear  operator.
		\item The map  (\textbf{synthesis operator})
		\begin{align*}
		\theta_\tau : \ell^p(\mathbb{N}) \ni \{a_n\}_n \mapsto \theta_\tau \{a_n\}_n\coloneqq \sum_{n=1}^\infty a_n\tau_n \in \mathcal{X}	
		\end{align*}
		is a well-defined bounded linear  operator.
	\end{enumerate}
Constants $a>0$ and $b>0$ satisfying 
\begin{align*}
a\|x\|\leq \|S_{f,\tau}	x\|\leq b \|x\|, \quad \forall x \in \mathcal{X},
\end{align*}
are called as \textbf{lower frame bound} and \textbf{upper frame bound}, respectively. If $S_{f,\tau}x=x$, $\forall x \in \mathcal{X}$, then we say that $ (\{f_n \}_{n}, \{\tau_n \}_{n}) $ is a Parseval p-ASF for $\mathcal{X}$. If $S_{f,\tau}x=\lambda x$, for some nonzero scalar $\lambda$, $\forall x \in \mathcal{X}$, then we say that $ (\{f_n \}_{n}, \{\tau_n \}_{n}) $ is a tight p-ASF for $\mathcal{X}$. If we do not demand  condition (i), then we say that $ (\{f_n \}_{n}, \{\tau_n \}_{n}) $ is a \textbf{p-approximate Bessel sequence (p-ABS)} for $\mathcal{X}$. Constant  $b$ is called as \textbf{Bessel bound}.
\end{definition}
Now we define the notion of p-orthonormal sequence and basis for  Banach spaces. Our basic motivation is the notion of orthonormal basis for Hilbert spaces and the standard canonical Schauder basis for $\ell^p(\mathbb{N})$.
\begin{definition}
Let  $p\in[1, \infty)$. Let $\mathbb{M} \subseteq\mathbb{N}$. 	Let $\{\tau_n\}_{n\in \mathbb{M}}$ be a sequence in  $\mathcal{X}$ and 	$\{f_n\}_{n\in \mathbb{M}}$ be a sequence in  $\mathcal{X}^*$. The pair $ (\{f_n \}_{n\in \mathbb{M}}, \{\tau_n \}_{n\in \mathbb{M}}) $ is said to be a \textbf{p-orthonormal sequence} for $\mathcal{X}$ if the following conditions are satisfied.	
	\begin{enumerate}[\upshape(i)]
		\item $f_n(\tau_m)=\delta_{n,m}$, $\forall n, m\in \mathbb{M}$. 
		\item For each $x\in \mathcal{X}$, 
		\begin{align*}
			\|x\|^p\geq \sum_{n\in \mathbb{M}}|f_n(x)|^p.
		\end{align*}
		\item For each $\{a_n\}_{n\in \mathbb{M}}\in \ell^p(\mathbb{M})$, 
		\begin{align*}
			\left\|\sum_{{n\in \mathbb{M}}} a_n \tau_n\right\|^p=\sum_{{n\in \mathbb{M}}}|a_n|^p.
		\end{align*}
	\end{enumerate}
\end{definition}
\begin{definition}\label{PRIESZ}
Let  $p\in[1, \infty)$.	Let $\{\tau_n\}_n$ be a sequence in  $\mathcal{X}$ and 	$\{f_n\}_n$ be a sequence in  $\mathcal{X}^*$.		The pair $ (\{f_n \}_{n}, \{\tau_n \}_{n}) $ is said to be a \textbf{p-orthonormal basis}   for $\mathcal{X}$ if the following conditions hold.
\begin{enumerate}[\upshape(i)]
	\item $\{\tau_n\}_n$ is  a Schauder basis for   $\mathcal{X}$.
	\item $f_n(\tau_m)=\delta_{n,m}$, $\forall n, m\in \mathbb{N}$, i.e., 	$\{f_n\}_n$ is the coordinate functionals associated with   $\{\tau_n\}_n$
	\item For each $x\in \mathcal{X}$, 
	\begin{align*}
		\|x\|^p=\sum_{n=1}^\infty|f_n(x)|^p.
	\end{align*}
	\item For each $\{a_n\}_n\in \ell^p(\mathbb{N})$, 
	\begin{align*}
		\left\|\sum_{n=1}^\infty a_n \tau_n\right\|^p=\sum_{n=1}^\infty|a_n|^p.
	\end{align*}
\end{enumerate}
\end{definition}
It is clear that a p-orthonormal basis is a p-orthonormal sequence and every subset of p-orthonormal basis is a p-orthonormal sequence. Also note that condition (iv) in Definition \ref{PRIESZ} says that if $ (\{f_n \}_{n}, \{\tau_n \}_{n}) $ is a p-orthonormal basis   for $\mathcal{X}$, then  $\{\tau_n\}_n$ is an unconditional  Schauder basis for   $\mathcal{X}$. Definition \ref{PRIESZ} give the following observations: 
\begin{align*}
	\|\tau_n\|=1=\|f_n\|, \quad \forall n \in \mathbb{N},
\end{align*}
\begin{align*}
	\|\tau_{j_1}+\cdots+\tau_{j_n}\|=n^\frac{1}{p}, \quad \forall j_1, \dots, j_n.
\end{align*}
\begin{remark}
	Let $\mathcal{H}$ be a Hilbert space and $\{\tau_n\}_n$ be an orthonormal basis for $\mathcal{H}$. Define $f_n: \mathcal{H} \ni h \mapsto \langle h, \tau_n \rangle \in \mathbb{K}$, $\forall n$. Then $ (\{f_n \}_{n}, \{\tau_n \}_{n}) $ is a 2-orthonormal basis   for $\mathcal{X}$.
\end{remark}
\begin{example}
	Standard Schauder basis and its co ordinate functionals is a p-orthonormal basis for $\ell^p(\mathbb{N})$, for each $p\in [1,\infty)$.
\end{example}
Gram-Schmidt orthonormalization converts every linearly independent set in a Hilbert space into an orthonormal set and hence to an orthonormal basis \cite{LEONBJORCKGANDER, DEBANTH}. On the other hand, there is a Gram-Schmidt orthonormalization  in Banach spaces due to Lin \cite{LINGRAMS}. We ask the following open problem. 
\begin{problem}
	\textbf{Let $\{\tau_j\}_{j=1}^n$ be a linearly independent collection in  $\mathcal{X}$ and 	$\{f_j\}_{j=1}^n$ be a linearly independent collection  in  $\mathcal{X}^*$. Whether there is a way to convert $ (\{f_j \}_{j=1}^n, \{\tau_j \}_{j=1}^n) $ into a p-orthonormal sequence, say $ (\{g_j \}_{j=1}^n, \{\omega_j \}_{j=1}^n) $ such that
	\begin{align*}
		\operatorname{span}\{f_1, \dots, f_j \}=	\operatorname{span}\{g_1, \dots, g_j \}~ \text{ and/or  } 	\operatorname{span}\{\tau_1, \dots, \tau_j \}=	\operatorname{span}\{\omega_1, \dots, \omega_j \}, \quad \forall 1\leq j \leq n.
	\end{align*}
In particular, whether we can convert a $ (\{f_j \}_{j=1}^n, \{\tau_j \}_{j=1}^n) $ into a p-orthonormal basis?}
\end{problem}
Like orthonormal basis for Hilbert spaces, we can characterize all p-orthonormal bases for Banach spaces.
\begin{theorem}\label{PREVIUOS}
Let $ (\{f_n \}_{n}, \{\tau_n \}_{n}) $ be a p-orthonormal basis   for $\mathcal{X}$. Then all p-orthonormal bases for $\mathcal{X}$ are precisely the family $ (\{f_nV^{-1} \}_{n}, \{V\tau_n \}_{n}) $, where $V:\mathcal{X}\to \mathcal{X}$ is an invertible isometry.
\end{theorem}
\begin{proof}
	(i) $\Rightarrow$ (ii) Let $ (\{g_n \}_{n}, \{\omega_n \}_{n}) $ be a p-orthonormal basis   for $\mathcal{X}$. Define 
	\begin{align*}
		V:\mathcal{X} \ni \sum_{n=1}^{\infty}f_n(x)\tau_n \mapsto \sum_{n=1}^{\infty}f_n(x)\omega_n \in \mathcal{X}
	\end{align*}
and 
	\begin{align*}
	U:\mathcal{X} \ni \sum_{n=1}^{\infty}g_n(x)\omega_n \mapsto \sum_{n=1}^{\infty}g_n(x)\tau_n \in \mathcal{X}.
\end{align*}
Then $U$ and $V$ are well-defined and 
\begin{align*}
UVx=U\left(\sum_{n=1}^{\infty}f_n(x)\omega_n\right)=\sum_{n=1}^{\infty}f_n(x)U\omega_n=\sum_{n=1}^{\infty}f_n(x)\tau_n=x, \quad \forall x \in \mathcal{X},
\end{align*}
\begin{align*}
VUx=V\left(\sum_{n=1}^{\infty}g_n(x)\tau_n\right)=\sum_{n=1}^{\infty}g_n(x)V \tau_n=\sum_{n=1}^{\infty}g_n(x)\omega_n=x, \quad \forall x \in \mathcal{X}	.
\end{align*}
Hence $V$ is invertible. Now 
\begin{align*}
	\|Vx\|=\left\|\sum_{n=1}^{\infty}f_n(x)\omega_n\right\|=\left(\sum_{n=1}^\infty|f_n(x)|^p\right)^\frac{1}{p}=\|x\|, \quad \forall x \in \mathcal{X}.
\end{align*}
 Hence $V$ is isometry. We now have $V\tau_n =\omega_n$, $\forall n \in \mathbb{N}$, and 
 \begin{align*}
 	f_n(V^{-1}x)=f_n(Ux)=f_n\left(\sum_{m=1}^{\infty}g_m(x)\tau_m\right)=\sum_{m=1}^{\infty}g_m(x)f_n(\tau_m)=g_n(x), \quad \forall n \in \mathbb{N}, \forall x \in \mathcal{X}.
 \end{align*}
(ii) $\Rightarrow$ (i) Define $g_n\coloneqq f_n V^{-1}, \omega_n\coloneqq V\tau_n$, $\forall n \in \mathbb{N}$. We claim that $ (\{g_n \}_{n}, \{\omega_n \}_{n}) $ is  a p-orthonormal basis   for $\mathcal{X}$. Since $V$ is invertible, $\{\omega_n\}_n$ is a Schauder basis for   $\mathcal{X}$. Now we verify remaining three conditions:
\begin{align*}
	g_n(\omega_m)=f_n(V^{-1}V\tau_m)=f_n(\tau_m)=\delta_{n,m}, \quad \forall n, m\in \mathbb{N},
\end{align*}
\begin{align*}
	\sum_{n=1}^\infty|g_n(x)|^p=	\sum_{n=1}^\infty|f_n(V^{-1}x)|^p=\|V^{-1}x\|^p=\|x\|^p, \quad \forall x \in \mathcal{X},
\end{align*}
\begin{align*}
	\left\|\sum_{n=1}^\infty a_n \omega_n\right\|^p&=	\left\|\sum_{n=1}^\infty a_n V\tau_n\right\|^p=\left\|V\left(\sum_{n=1}^\infty a_n \tau_n\right)\right\|^p\\
	&=\left\|\sum_{n=1}^\infty a_n \tau_n\right\|^p=\sum_{n=1}^\infty|a_n|^p, \quad 	\{a_n\}_n\in \ell^p(\mathbb{N}).
\end{align*}
\end{proof}
\begin{theorem}
If 	 $\mathcal{X}$ admits a p-orthonormal basis, then $\mathcal{X}$ is isometrically isomorphic to $\ell^p(\mathbb{N})$.
\end{theorem}
\begin{proof}
Let $ (\{f_n \}_{n}, \{\tau_n \}_{n}) $ be a p-orthonormal basis   for $\mathcal{X}$.  Define 
\begin{align*}
	V:\mathcal{X} \ni \sum_{n=1}^{\infty}f_n(x)\tau_n \mapsto \sum_{n=1}^{\infty}f_n(x)e_n \in \ell^p(\mathbb{N})
\end{align*}
and 
\begin{align*}
	U:\ell^p(\mathbb{N}) \ni \sum_{n=1}^{\infty}\zeta_n(x)e_n \mapsto \sum_{n=1}^{\infty}\zeta_n(x)\tau_n \in \mathcal{X}.
\end{align*}	
Other parts are similar to the proof of Theorem \ref{PREVIUOS}.
\end{proof}
\begin{theorem}
If a Banach space $\mathcal{X}$ is isometrically isomorphic to $\ell^p(\mathbb{N})$, then it admits a p-orthonormal basis.	
\end{theorem}
\begin{proof}
	Let $V: \mathcal{X} \to \ell^p(\mathbb{N})$ be an isometric isomorphism. Define $f_n\coloneqq \zeta _n V, \tau_n\coloneqq V^{-1}e_n$, $\forall n \in \mathbb{N}$. Then  $ (\{f_n \}_{n}, \{\tau_n \}_{n}) $ is  a p-orthonormal basis   for $\mathcal{X}$. 
\end{proof}
Now we want to define the notion of Riesz basis in accordance with Definition \ref{RIESZ}. We note that Riesz basis notion has been defined for Banach spaces  using a single sequence in \cite{CHRISTENSENSTOEVA2003, ALDROUBI}  which we do not consider in this paper. Our motivation comes from the definition given in \cite{MAHESHJOHNSON2}.
\begin{definition}\label{PRIESZDEFINITION}
Let  $p\in[1, \infty)$.	Let $\{\tau_n\}_n$ be a sequence in  $\mathcal{X}$ and 	$\{f_n\}_n$ be a sequence in  $\mathcal{X}^*$.	The pair $ (\{f_n \}_{n}, \{\tau_n \}_{n}) $ is said to be a \textbf{p-approximate Riesz basis}   for $\mathcal{X}$ if there exist bounded linear invertible operators $U,V:\mathcal{X}\to \mathcal{X}$ such that 
	\begin{align*}
		f_n=g_n U,~	\tau_n=V\omega_n, \quad \forall n \in \mathbb{N}, 
	\end{align*}
where $(\{g_n\}_n, \{\omega_n\}_n)$  is a p-orthonormal basis    for $\mathcal{X}$.
\end{definition}
\begin{definition}
Let  $p\in[1, \infty)$. Let $\{\tau_n\}_n$ be a sequence in  $\mathcal{X}$ and 	$\{f_n\}_n$ be a sequence in  $\mathcal{X}^*$.	The pair $ (\{f_n \}_{n}, \{\tau_n \}_{n}) $ is said to be a \textbf{p-approximate Riesz sequence}    for $\mathcal{X}$ if	$ (\{f_n \}_{n}, \{\tau_n \}_{n}) $ is a  p-approximate Riesz basis    for $\overline{span}\{\tau_n\}_n$. Constants $a,b>0$ satisfying  for all $m\in \mathbb{N}$, 
\begin{align*}
	a\left(\sum_{n=1}^{m}|c_n|^p \right)^\frac{1}{p}\leq \left\|\sum_{n=1}^{m}c_n\tau_n\right\|\leq b\left(\sum_{n=1}^{m}|c_n|^p \right)^\frac{1}{p},\quad  \forall c_1, \dots, c_m \in \mathbb{K}
\end{align*}
are called as \textbf{lower p-approximate Riesz bound} and \textbf{upper p-approximate Riesz  bound}, respectively. Note that they exist.  We say that $ (\{f_n \}_{n}, \{\tau_n \}_{n}) $ is unit norm if $\|f_n\|=\|\tau_n\|=|f_n(\tau_n)|=1$, $\forall n\in \mathbb{N}$.
\end{definition}
 Note that Definition  \ref{PRIESZDEFINITION} of p-approximate Riesz basis is not the same as the definition of p-approximate Riesz basis given in \cite{MAHESHJOHNSON2}. The following theorem says that they are equivalent.  
 In the paper, given a space $\mathcal{X}$, by $I_\mathcal{X}$ we mean the identity operator on $\mathcal{X}$.
  \begin{theorem}\label{RIESZCHAR}
  	The following are equivalent.
  	\begin{enumerate}[\upshape(i)]
  		\item The pair $ (\{f_n \}_{n}, \{\tau_n \}_{n}) $ is  a p-approximate Riesz basis   for $\mathcal{X}$.
  		\item The pair $ (\{f_n \}_{n}, \{\tau_n \}_{n}) $ is  a p-ASF   for $\mathcal{X}$ and 
  		\begin{align}\label{RIESZEQUATION}
  		\theta_fS_{f,\tau}^{-1}\theta_\tau=I_{\ell^p(\mathbb{N})}.	
  		\end{align}
  	\end{enumerate}
  \end{theorem}
\begin{proof}
	\begin{enumerate}[\upshape(i)]
	\item $\Rightarrow$ (ii) Let $U,V:\mathcal{X}\to \mathcal{X}$ be bounded linear invertible operators such that $	f_n=g_n U,~	\tau_n=V\omega_n,  \forall n \in \mathbb{N}$,  where $(\{g_n\}_n, \{\omega_n\}_n)$  is a p-orthonormal basis    for $\mathcal{X}$. We then have $\theta_f=\theta_gU$ and $\theta_\tau=V\theta_\omega$. Hence $S_{f,\tau}=\theta_\tau \theta_f=V\theta_\omega \theta_gU=VU$ which is invertible. Therefore $ (\{f_n \}_{n}, \{\tau_n \}_{n}) $ is  a p-ASF   for $\mathcal{X}$. We now find 
	\begin{align*}
		\theta_fS_{f,\tau}^{-1}\theta_\tau\{a_n\}_n&=\theta_gU(VU)^{-1}V\theta_\omega\{a_n\}_n=\theta_gUU^{-1}V^{-1}V\theta_\omega\{a_n\}_n\\
		&=\theta_g \theta_\omega\{a_n\}_n=\theta_g \left(\sum_{n=1}^{\infty}a_n\omega_n\right)=\sum_{n=1}^{\infty}a_n\theta_g\omega_n\\
		&=\sum_{n=1}^{\infty}a_n\sum_{m=1}^{\infty}g_m(\omega_n)e_m= \sum_{n=1}^{\infty}a_ne_n=\{a_n\}_n,\quad \forall \{a_n\}_n \in \ell^p(\mathbb{N}).		
	\end{align*}
	\item $\Rightarrow$ (i) Let $(\{g_n\}_n, \{\omega_n\}_n)$  be  a p-orthonormal basis    for $\mathcal{X}$. Define $U\coloneqq \theta_\omega\theta_f$ and $V\coloneqq \theta_\tau \theta_g$. Then 
	
	\begin{align*}
		g_n(Ux)&=g_n(\theta_\omega \theta_fx)=g_n\left(\theta_\omega\left(\sum_{m=1}^{\infty}f_m(x)e_m\right)\right)\\
		&=g_n\left(\sum_{m=1}^{\infty}f_m(x)\theta_\omega e_m\right)
		=g_n\left(\sum_{m=1}^{\infty}f_m(x) \omega_m\right)\\
		&=\sum_{m=1}^{\infty}f_m(x) g_n(\omega_m)=f_n(x), \quad \forall x \in \mathcal{X},
			\end{align*}
			\begin{align*}
		V\omega_n=\theta_\tau \theta_g\omega_n=\theta_\tau\left(\sum_{m=1}^\infty g_m(\omega_n)e_m\right)=\theta_\tau e_n=\tau_n, \quad  \forall n \in \mathbb{N}.
	\end{align*}
To complete the proof, we have to show that both $U$ and $V$ are bounded invertible. First note that 
\begin{align*}
	\theta_g\theta_\omega\{a_n\}_n&=	\theta_g \left(\sum_{n=1}^{\infty}a_n\omega_n\right)=\sum_{n=1}^{\infty}a_n\theta_g\omega_n
	\\
	&=\sum_{n=1}^{\infty}a_n\sum_{m=1}^{\infty}g_m(\omega_n)e_m
	= \sum_{n=1}^{\infty}a_ne_n\\
	&=\{a_n\}_n, \quad \forall \{a_n\}_n \in \ell^p(\mathbb{N}).
\end{align*} 
Hence 
\begin{align}\label{RIESZEPP}
	\theta_g\theta_\omega =I_{\ell^p(\mathbb{N})}.
\end{align} 
Now using Equation (\ref{RIESZEQUATION}) and  Equation (\ref{RIESZEPP}) we find 
\begin{align*}
	&U (S_{f,\tau}^{-1}\theta_\tau\theta_g)=\theta_\omega \theta_fS_{f,\tau}^{-1}\theta_\tau\theta_g=\theta_\omega I_{\ell^p(\mathbb{N})}\theta_g=\theta_\omega \theta_g=I_\mathcal{X},\\
	&(S_{f,\tau}^{-1}\theta_\tau\theta_g)U=S_{f,\tau}^{-1}\theta_\tau\theta_g\theta_\omega \theta_f=S_{f,\tau}^{-1}\theta_\tau I_{\ell^p(\mathbb{N})}\theta_f=S_{f,\tau}^{-1}\theta_\tau \theta_f=I_\mathcal{X}
\end{align*}
and 
\begin{align*}
	&V(\theta_\omega \theta_fS_{f,\tau}^{-1})=\theta_\tau \theta_g\theta_\omega \theta_fS_{f,\tau}^{-1}=\theta_\tau I_{\ell^p(\mathbb{N})}\theta_fS_{f,\tau}^{-1}=\theta_\tau\theta_fS_{f,\tau}^{-1}=I_\mathcal{X},\\
	&(\theta_\omega \theta_fS_{f,\tau}^{-1})V=\theta_\omega \theta_fS_{f,\tau}^{-1}\theta_\tau \theta_g=\theta_\omega I_{\ell^p(\mathbb{N})}\theta_g=\theta_\omega\theta_g=I_\mathcal{X}.
\end{align*}
Hence $U$ and $V$ are invertible.
\end{enumerate}	
\end{proof}
Now we formulate Feichtinger conjectures for Banach spaces.
\begin{conjecture}\label{FB}\textbf{(Feichtinger conjecture for p-approximate Schauder frames)
	Let  $ (\{f_n \}_{n}, $ $ \{\tau_n \}_{n}) $  be  a p-ASF for   $\mathcal{X}$ such that 
\begin{align*}
	&0<\inf_{n\in \mathbb{N}}\|\tau_n\|\leq \sup_{n\in \mathbb{N}}\|\tau_n\|<\infty ,\\
	&0<\inf_{n\in \mathbb{N}}\|f_n\|\leq \sup_{n\in \mathbb{N}}\|f_n\|<\infty.
\end{align*}
Then $ (\{f_n \}_{n}, \{\tau_n \}_{n}) $  can be partitioned into a finite union of p-approximate Riesz sequences. Moreover, what is the number of partitions required?}
\end{conjecture}
Note that using Feichtinger conjecture (result of Marcus, Spielman, and Srivastava) for Hilbert spaces we get the following. Let $ (\{f_n \}_{n}, $ $ \{\tau_n \}_{n}) $  be  a  2-ASF for a Hilbert space   $\mathcal{H}$ such that 
\begin{align*}
	&0<\inf_{n\in \mathbb{N}}\|\tau_n\|\leq \sup_{n\in \mathbb{N}}\|\tau_n\|<\infty ,\\
	&0<\inf_{n\in \mathbb{N}}\|f_n\|\leq \sup_{n\in \mathbb{N}}\|f_n\|<\infty.
\end{align*}
 Riesz representation theorem says we can identify  $ \{f_n \}_{n} $ by  a sequence $ \{\omega_n \}_{n} $ in $\mathcal{H}$. Therefore, we have $ \{\tau_n \}_{n} $ and $ \{\omega_n \}_{n} $ are frames for $\mathcal{H}$.
Feichtinger conjecture for Hilbert spaces says that both $ \{\tau_n \}_{n} $ and $ \{\omega_n \}_{n} $ can be partitioned into  a finite union of  Riesz sequences. However, this does not solve Conjecture \ref{FB} even in this special  case. The reason is that the number of partitions of $ \{\tau_n \}_{n} $ and $ \{\omega_n \}_{n} $ may be different. We want a partition  of the form 
\begin{align*}
 (\{f_n \}_{n},  \{\tau_n \}_{n})= (\{f^{(1)}_n \}_{n}, \{\tau^{(1)}_n \}_{n})	\cup \cdots \cup (\{f^{(m)}_n \}_{n}, \{\tau^{(m)}_n \}_{n}), 
\end{align*}
where each of $(\{f^{(1)}_n \}_{n}, \{\tau^{(1)}_n \}_{n}),  \dots,  (\{f^{(m)}_n \}_{n}, \{\tau^{(m)}_n \}_{n})$ is a p-approximate Riesz sequence.
\begin{conjecture}\label{FS}\textbf{(Feichtinger conjecture for p-approximate Bessel sequences)
		Let  $ (\{f_n \}_{n}, $ $ \{\tau_n \}_{n}) $  be  a p-ABS for   $\mathcal{X}$ such that 
		\begin{align*}
			&0<\inf_{n\in \mathbb{N}}\|\tau_n\|\leq \sup_{n\in \mathbb{N}}\|\tau_n\|<\infty ,\\
		&0<\inf_{n\in \mathbb{N}}\|f_n\|\leq \sup_{n\in \mathbb{N}}\|f_n\|<\infty.
		\end{align*}
		Then $ (\{f_n \}_{n}, \{\tau_n \}_{n}) $  can be partitioned into a finite union of p-approximate Riesz sequences. Moreover, what is the number of partitions required?}
\end{conjecture}
\begin{conjecture}\textbf{(Finite dimensional Feichtinger conjecture for p-approximate Schauder frames)}\label{3}
\textbf{Let 	 $\mathcal{X}$ be a $d$-dimensional Banach  space. For every    real $b,c>0$, there exist a natural number $M(b,c)$,  a real $a(b,c)>0$ so that whenever $(\{f_j\}_{j=1}^n, \{\tau_j\}_{j=1}^n)$ is a p-ASF   for   $\mathcal{X}$ with upper frame bound $b$ and
	\begin{align*}
	\|\tau_j\|\geq c \quad \text{ and }\quad 	\|f_j\|\geq c, \quad \forall 1\leq j \leq n,
	\end{align*}
 then the set $\{1,2, \dots, n\}$ can be partitioned into sets $I_1,I_2,\dots,  I_{M(b,c)}$ so that for each $1\leq k \leq M(b,c)$, $ (\{f_j\}_{j\in I_k}, \{\tau_j\}_{j\in I_k})$ is a p-approximate Riesz sequence with lower p-approximate Riesz bound  $a(b,c)$ and upper p-approximate Riesz bound $b$. Moreover, what is the number of partitions required?}	
\end{conjecture}
\begin{conjecture}\textbf{(Finite dimensional Feichtinger conjecture for p-approximate Bessel sequences)}\label{4}
	\textbf{Let 	 $\mathcal{X}$ be a $d$-dimensional Banach space. For every $b>0$, there exists a natural number $M(b)$ and a real  $a(b)>0$ so that for every p-ABS $(\{f_j\}_{j=1}^n, \{\tau_j\}_{j=1}^n)$ for $\mathcal{X}$  with Bessel bound $b$ and 
			\begin{align*}
			\|\tau_j\|=1 \quad \text{ and }\quad 	\|f_j\|=1, \quad \forall 1\leq j \leq n
		\end{align*}
	 can be written as a union of $M(b)$ p-approximate Riesz basic sequences each with lower p-approximate Riesz bound $a(b)$. Moreover, what is the number of partitions required?}
\end{conjecture}
Casazza-Vershynin conjecture \cite{CASAZZABOOK} known as $R_\varepsilon$-conjecture for Hilbert spaces reads as follows.
\begin{theorem}\cite{CASAZZABOOK}
\textbf{(Casazza-Vershynin conjecture or $R_\epsilon$-conjecture for Hilbert spaces)}	\textbf{For every $\varepsilon>0$, every unit norm  Riesz sequence can be written as a finite union of unit norm $\varepsilon$-Riesz sequences}.
\end{theorem}
We next formulate Casazza-Vershynin conjecture \cite{CASAZZABOOK}/$R_\varepsilon$-conjecture for Banach spaces. First we need a definition. 
\begin{definition}
Let  $p\in[1, \infty)$. A p-approximate Riesz sequence $ (\{f_n \}_{n}, \{\tau_n \}_{n}) $     for $\mathcal{X}$ is said to be \textbf{unit norm $\varepsilon$-p-approximate Riesz sequence} ($\varepsilon<1$) if the following conditions hold.
\begin{enumerate}[\upshape(i)]
	\item $\|f_n\|=\|\tau_n\|=|f_n(\tau_n)|=1$, $\forall n\in \mathbb{N}$.
	\item  For all $m\in \mathbb{N}$, 
	\begin{align*}
		(1-\varepsilon)\left(\sum_{n=1}^{m}|c_n|^p \right)^\frac{1}{p}\leq \left\|\sum_{n=1}^{m}c_n\tau_n\right\|\leq (1+\varepsilon)\left(\sum_{n=1}^{m}|c_n|^p \right)^\frac{1}{p},\quad  \forall c_1, \dots, c_m \in \mathbb{K}.
	\end{align*}
\end{enumerate}   
\end{definition}
\begin{conjecture}\label{11}\textbf{(Casazza-Vershynin conjecture or $R_\epsilon$-conjecture for Banach spaces}) \textbf{For every $\varepsilon>0$, every unit norm p-approximate Riesz sequence can be written as a finite union of unit norm $\varepsilon$-p-approximate Riesz sequences}.
\end{conjecture}
Weaver's conjecture for Hilbert spaces states as follows. 
\begin{theorem}\cite{WEAVER2004}
\textbf{(Weaver's conjecture for Hilbert spaces)}
\textbf{Let  $\mathcal{H}$ be a $d$-dimensional Hilbert  space.	There are universal constants $b\geq 2$ and $ b>\varepsilon > 0$ such that the following holds. Let $\{\tau_j\}_{j=1}^n$ be a collection in  $\mathcal{H}$ satisfying: 
	\begin{align*}
		\|\tau_j\|\leq 1,  \quad \forall 1\leq j \leq n
	\end{align*}
	and 
	\begin{align*}
		\sum_{j=1}^{n}|\langle h, \tau_j\rangle|^2\leq b\|h\|^2, \quad \forall h \in \mathcal{H}.
	\end{align*}
	Then there exists a partition $I_1, \dots, I_M$ of $\{1, 2, \dots, n\}$ such that 
	\begin{align*}
		\sum_{j\in I_k}|\langle h, \tau_j\rangle|^2\leq (b-\varepsilon)\|h\|^2, \quad \forall h \in \mathcal{H}, \forall 1\leq k \leq M.
\end{align*} }	
\end{theorem}
Here is Weaver's conjecture  \cite{WEAVER2004} for Banach spaces.
\begin{conjecture}\label{12}(\textbf{Weaver's conjecture for Banach spaces - 1})
\textbf{Let  $\mathcal{X}$ be a $d$-dimensional Banach  space.	There are universal constants $b\geq 2$ and $ b>\varepsilon > 0$ such that the following holds. Let $\{\tau_j\}_{j=1}^n$ be a collection in  $\mathcal{X}$ and  $\{f_j\}_{j=1}^n$ be a collection in  $\mathcal{X}^*$ satisfying: 
\begin{align*}
	\|f_j\|\leq 1, ~ \|\tau_j\|\leq 1, \quad \forall 1\leq j \leq n
\end{align*}
and 
\begin{align*}
	\left\|\sum_{j=1}^{n}f_j(x)\tau_j\right\|\leq b\|x\|, \quad \forall x \in \mathcal{X}.
\end{align*}
Then there exists a partition $I_1, \dots, I_M$ of $\{1, 2, \dots, n\}$ such that 
\begin{align*}
	\left\|\sum_{j\in I_k}f_j(x)\tau_j\right\|\leq (b-\varepsilon)\|x\|, \quad \forall x \in \mathcal{X}, \forall 1\leq k \leq M.
\end{align*} }
\end{conjecture}
Well, we can make variants of Conjecture \ref{12} as follows. 
\begin{conjecture}\label{121}
	(\textbf{Weaver's conjecture for Banach spaces - 2})
\textbf{Let  $\mathcal{X}$ be a $d$-dimensional Banach  space.	There are universal constants $b\geq 2$ and $ b>\varepsilon > 0$ such that the following holds. Let $\{\tau_j\}_{j=1}^n$ be a collection in  $\mathcal{X}$ and  $\{f_j\}_{j=1}^n$ be a collection in  $\mathcal{X}^*$ satisfying: 
	\begin{align*}
		\|f_j\|\leq 1, ~ \|\tau_j\|\leq 1, \quad \forall 1\leq j \leq n
	\end{align*}
	and 
	\begin{align*}
		\left\|\sum_{j=1}^{n}f_j(x)\tau_j\right\|= b\|x\|, \quad \forall x \in \mathcal{X}.
	\end{align*}
	Then there exists a partition $I_1, \dots, I_M$ of $\{1, 2, \dots, n\}$ such that 
	\begin{align*}
		\left\|\sum_{j\in I_k}f_j(x)\tau_j\right\|\leq (b-\varepsilon)\|x\|, \quad \forall x \in \mathcal{X}, \forall 1\leq k \leq M.
\end{align*} }	
\end{conjecture}
\begin{conjecture}\label{122}
(\textbf{Weaver's conjecture for Banach spaces - 3})	
\textbf{Let  $\mathcal{X}$ be a $d$-dimensional Banach  space.	There are universal constants $b\geq 2$ and $ b>\varepsilon > 0$ such that the following holds. Let $\{\tau_j\}_{j=1}^n$ be a collection in  $\mathcal{X}$ and  $\{f_j\}_{j=1}^n$ be a collection in  $\mathcal{X}^*$ satisfying: 
	\begin{align*}
		\|f_j\|= \|\tau_j\|=|f_j(\tau_j)|=1, \quad \forall 1\leq j \leq n
	\end{align*}
	and 
	\begin{align*}
		\left\|\sum_{j=1}^{n}f_j(x)\tau_j\right\|\leq  b\|x\|, \quad \forall x \in \mathcal{X}.
	\end{align*}
	Then there exists a partition $I_1, \dots, I_M$ of $\{1, 2, \dots, n\}$ such that 
	\begin{align*}
		\left\|\sum_{j\in I_k}f_j(x)\tau_j\right\|\leq (b-\varepsilon)\|x\|, \quad \forall x \in \mathcal{X}, \forall 1\leq k \leq M.
\end{align*} }	
\end{conjecture}
\begin{conjecture}\label{123}
(\textbf{Weaver's conjecture for Banach spaces - 4})	
\textbf{Let  $\mathcal{X}$ be a $d$-dimensional Banach  space.	There are universal constants $b\geq 2$ and $ b>\varepsilon > 0$ such that the following holds. Let $\{\tau_j\}_{j=1}^n$ be a collection in  $\mathcal{X}$ and  $\{f_j\}_{j=1}^n$ be a collection in  $\mathcal{X}^*$ satisfying: 
	\begin{align*}
		\|f_j\|= \|\tau_j\|=|f_j(\tau_j)|=1, \quad \forall 1\leq j \leq n
	\end{align*}
	and 
	\begin{align*}
		\left\|\sum_{j=1}^{n}f_j(x)\tau_j\right\|=  b\|x\|, \quad \forall x \in \mathcal{X}.
	\end{align*}
	Then there exists a partition $I_1, \dots, I_M$ of $\{1, 2, \dots, n\}$ such that 
	\begin{align*}
		\left\|\sum_{j\in I_k}f_j(x)\tau_j\right\|\leq (b-\varepsilon)\|x\|, \quad \forall x \in \mathcal{X}, \forall 1\leq k \leq M.
\end{align*} }	
\end{conjecture}
Previous four Weaver's conjectures for Banach spaces can be stated in slightly stronger form as follows.
\begin{conjecture}\label{1231}
(\textbf{Weaver's conjecture for Banach spaces - strong form - 1})
\textbf{Let  $\mathcal{X}$ be a $d$-dimensional Banach  space.	There are universal constants $b\geq 2$ and $ b>\varepsilon > 0$ such that the following holds. Let $\{\tau_j\}_{j=1}^n$ be a collection in  $\mathcal{X}$ and  $\{f_j\}_{j=1}^n$ be a collection in  $\mathcal{X}^*$ satisfying: 
	\begin{align*}
		\|f_j\|\leq 1, ~ \|\tau_j\|\leq 1, \quad \forall 1\leq j \leq n, 
	\end{align*}
the spectrum of the  operator $S_{f,\tau}:\mathcal{X} \ni x \mapsto \sum_{j=1}^{n}f_j(x)\tau_j \in \mathcal{X}$  lies in $[0,\infty)$ 
	and 
	\begin{align*}
		\left\|\sum_{j=1}^{n}f_j(x)\tau_j\right\|\leq b\|x\|, \quad \forall x \in \mathcal{X}.
	\end{align*}
	Then there exists a partition $I_1, \dots, I_M$ of $\{1, 2, \dots, n\}$ such that 
	\begin{align*}
		\left\|\sum_{j\in I_k}f_j(x)\tau_j\right\|\leq (b-\varepsilon)\|x\|, \quad \forall x \in \mathcal{X}, \forall 1\leq k \leq M.
\end{align*} }	
\end{conjecture}
\begin{conjecture}\label{1232}
(\textbf{Weaver's conjecture for Banach spaces - strong form - 2})	
\textbf{Let  $\mathcal{X}$ be a $d$-dimensional Banach  space.	There are universal constants $b\geq 2$ and $ b>\varepsilon > 0$ such that the following holds. Let $\{\tau_j\}_{j=1}^n$ be a collection in  $\mathcal{X}$ and  $\{f_j\}_{j=1}^n$ be a collection in  $\mathcal{X}^*$ satisfying: 
	\begin{align*}
		\|f_j\|\leq1, ~ \|\tau_j\|\leq 1, \quad \forall 1\leq j \leq n, 
	\end{align*}
the spectrum of the  operator $S_{f,\tau}:\mathcal{X} \ni x \mapsto \sum_{j=1}^{n}f_j(x)\tau_j \in \mathcal{X}$  lies in $[0,\infty)$ 
	and 
	\begin{align*}
		\left\|\sum_{j=1}^{n}f_j(x)\tau_j\right\|=  b\|x\|, \quad \forall x \in \mathcal{X}.
	\end{align*}
	Then there exists a partition $I_1, \dots, I_M$ of $\{1, 2, \dots, n\}$ such that 
	\begin{align*}
		\left\|\sum_{j\in I_k}f_j(x)\tau_j\right\|\leq (b-\varepsilon)\|x\|, \quad \forall x \in \mathcal{X}, \forall 1\leq k \leq M.
\end{align*} }		
\end{conjecture}
\begin{conjecture}\label{1233}
(\textbf{Weaver's conjecture for Banach spaces - strong form - 3})	
\textbf{Let  $\mathcal{X}$ be a $d$-dimensional Banach  space.	There are universal constants $b\geq 2$ and $ b>\varepsilon > 0$ such that the following holds. Let $\{\tau_j\}_{j=1}^n$ be a collection in  $\mathcal{X}$ and  $\{f_j\}_{j=1}^n$ be a collection in  $\mathcal{X}^*$ satisfying: 
	\begin{align*}
		\|f_j\|= \|\tau_j\|=|f_j(\tau_j)|=1, \quad \forall 1\leq j \leq n,
	\end{align*}
the spectrum of the  operator $S_{f,\tau}:\mathcal{X} \ni x \mapsto \sum_{j=1}^{n}f_j(x)\tau_j \in \mathcal{X}$  lies in $[0,\infty)$ 	and 
	\begin{align*}
		\left\|\sum_{j=1}^{n}f_j(x)\tau_j\right\|\leq   b\|x\|, \quad \forall x \in \mathcal{X}.
	\end{align*}
	Then there exists a partition $I_1, \dots, I_M$ of $\{1, 2, \dots, n\}$ such that 
	\begin{align*}
		\left\|\sum_{j\in I_k}f_j(x)\tau_j\right\|\leq (b-\varepsilon)\|x\|, \quad \forall x \in \mathcal{X}, \forall 1\leq k \leq M.
\end{align*} }		
\end{conjecture}
\begin{conjecture}\label{1234}
(\textbf{Weaver's conjecture for Banach spaces - strong form - 4})	
\textbf{Let  $\mathcal{X}$ be a $d$-dimensional Banach  space.	There are universal constants $b\geq 2$ and $ b>\varepsilon > 0$ such that the following holds. Let $\{\tau_j\}_{j=1}^n$ be a collection in  $\mathcal{X}$ and  $\{f_j\}_{j=1}^n$ be a collection in  $\mathcal{X}^*$ satisfying: 
	\begin{align*}
		\|f_j\|= \|\tau_j\|=|f_j(\tau_j)|=1, \quad \forall 1\leq j \leq n,
	\end{align*}
the spectrum of the  operator $S_{f,\tau}:\mathcal{X} \ni x \mapsto \sum_{j=1}^{n}f_j(x)\tau_j \in \mathcal{X}$  lies in $[0,\infty)$ 	and 
	\begin{align*}
		\left\|\sum_{j=1}^{n}f_j(x)\tau_j\right\|=  b\|x\|, \quad \forall x \in \mathcal{X}.
	\end{align*}
	Then there exists a partition $I_1, \dots, I_M$ of $\{1, 2, \dots, n\}$ such that 
	\begin{align*}
		\left\|\sum_{j\in I_k}f_j(x)\tau_j\right\|\leq (b-\varepsilon)\|x\|, \quad \forall x \in \mathcal{X}, \forall 1\leq k \leq M.
\end{align*} }		
\end{conjecture}
Here are  Akemann-Weaver conjectures  for Banach space (see \cite{BOWNIK2018PRO, AKEMANNWEAVER2014} for Hilbert spaces).
\begin{conjecture}\label{AW}
(\textbf{Akemann-Weaver conjecture for Banach spaces}) \textbf{ Let  $\mathcal{X}$ be a  Banach  space (finite or infinite dimensional). There exists a universal constant $c$  such that the following holds. 	Let 	$ (\{f_n \}_{n}, \{\tau_n \}_{n}) $ be a p-approximate Bessel sequence  with Bessel bound 1  for $\mathcal{X}$ satisfying
	\begin{align*}
	\sup_{n\in \mathbb{N}}\|f_n\|^q\leq \varepsilon, \quad 	\sup_{n\in \mathbb{N}}\|\tau_n\|^s\leq \varepsilon,
	\end{align*}
 for some $\epsilon>0$, for some $q,s>0$. Let $\{r_n \}_{n}$ be any sequence in [0,1]. Then there exists a subset $\mathbb{M} \subseteq \mathbb{N}$ and a $d>0$  such that 
\begin{align*}
	\left\| \sum_{n\in \mathbb{M}}f_n(x)\tau_n-\sum_{n=1}^{\infty}r_nf_n(x)\tau_n\right\|\leq c \varepsilon^\frac{1}{d}\|x\|, \quad \forall x \in \mathcal{X}.
\end{align*}
i.e., 
\begin{align*}
	\left\| \sum_{n\in \mathbb{M}}f_n(\cdot)\tau_n-\sum_{n=1}^{\infty}r_nf_n(\cdot)\tau_n\right\|\leq c \varepsilon^\frac{1}{d}.
\end{align*}}
\end{conjecture}
\begin{conjecture}\label{AW2}
(\textbf{Akemann-Weaver conjecture for Banach spaces - strong form})		\textbf{ Let  $\mathcal{X}$ be a Banach  space (finite or infinite dimensional). There exists a universal constant $c$  such that the following holds. 	Let 	$ (\{f_n \}_{n}, \{\tau_n \}_{n}) $ be a p-approximate Bessel sequence with Bessel bound 1   for $\mathcal{X}$ satisfying 
	\begin{align*}
		\sup_{n\in \mathbb{N}}\|f_n\|^q\leq \varepsilon, \quad 	\sup_{n\in \mathbb{N}}\|\tau_n\|^s\leq \varepsilon,
	\end{align*}
 for some $\epsilon>0$, for some $q,s>0$ and the spectrum of the  operator $S_{f,\tau}:\mathcal{X} \ni x \mapsto \sum_{n=1}^{\infty}f_n(x)\tau_n \in \mathcal{X}$  lies in $[0,\infty)$. Let $\{r_n \}_{n}$ be any sequence in [0,1]. Then there exists a subset $\mathbb{M} \subseteq \mathbb{N}$ and a $d>0$  such that 
	\begin{align*}
		\left\| \sum_{n\in \mathbb{M}}f_n(x)\tau_n-\sum_{n=1}^{\infty}r_nf_n(x)\tau_n\right\|\leq c \varepsilon^\frac{1}{d}\|x\|, \quad \forall x \in \mathcal{X},
\end{align*}
i.e., 
\begin{align*}
	\left\| \sum_{n\in \mathbb{M}}f_n(\cdot)\tau_n-\sum_{n=1}^{\infty}r_nf_n(\cdot)\tau_n\right\|\leq c \varepsilon^\frac{1}{d}.
\end{align*}}
\end{conjecture}
  We now formulate three more conjectures which are motivated from   results of Casazza  \cite{CASAZZATHREE} stated as follows.
  \begin{theorem}\cite{CASAZZATHREE}
  	\textbf{Every Riesz basis for a Hilbert space can be written as a  linear combination of two orthonormal bases.}
  \end{theorem}
 \begin{theorem}\cite{CASAZZATHREE}
 	\textbf{Every frame for a Hilbert space is a multiple of a sum of three orthonormal basis.}
 \end{theorem}
\begin{theorem}\cite{CASAZZATHREE}
	\textbf{Every frame for a Hilbert space is a multiple of a sum of an  orthonormal basis and a Riesz basis.}
\end{theorem}
 \begin{conjecture}\label{13}
	\textbf{Every p-approximate Riesz basis  is a linear combination of  p-orthonormal bases. Moreover, there exists $M \in \mathbb{N}$ such that every  p-approximate Riesz basis  can be written as a linear combination of   $M$ p-orthonormal bases.}
\end{conjecture}

  \begin{conjecture}\label{14}
  	\textbf{Every p-ASF is a multiple of a  finite sum of p-orthonormal bases. Moreover, there exists $M \in \mathbb{N}$ such that every  p-ASF can be written as a multiple of   sum of $M$ p-orthonormal bases.}
  \end{conjecture}
 \begin{conjecture}\label{141}
	\textbf{Every p-ASF is a multiple of a   sum of p-orthonormal basis and a p-approximate Riesz basis. }
\end{conjecture}
We next recall the famous fundamental inequality for  finite frames for finite dimensional Hilbert spaces due to Casazza, Fickus, Kovacevic, Leon, and Tremain \cite{CSAZZAFUNDAMENTAL} and independently by Viswanath and Anantharam \cite{VISWANATHANATHARAM}.
 \begin{theorem}\cite{CSAZZAFUNDAMENTAL, VISWANATHANATHARAM}
(\textbf{Fundamental inequality})	\textbf{Let $n\geq d$ and $\mathcal{H}$ be a $d$-dimensional Hilbert   space and let $a_1, \dots, a_n$ be a collection of non negative real numbers. Then there exists a tight frame  $\{\tau_j\}_{j=1}^{n}$  for $\mathcal{H}$ with $\|\tau_j\|=a_j$, $\forall 1\leq j\leq n$ if and only if 
	\begin{align*}
		\max_{1\leq j \leq n}a_j^2\leq \frac{1}{d}\sum_{j=1}^{n}a_j^2.
	\end{align*}}
\end{theorem} 
Here is fundamental inequality conjecture for Banach spaces.
\begin{conjecture}\label{FUNDAMENTALINEQUALITY}
	(\textbf{Fundamental inequality conjecture for Banach spaces}) \textbf{Let $n\geq d$ and $\mathcal{X}$ be a $d$-dimensional Banach  space and let $a_1, \dots, a_n, b_1, \dots, b_n, c_1, \dots, c_n$ be a collection of non negative real numbers. Then there exists a tight ASF $(\{f_j\}_{j=1}^n, \{\tau_j\}_{j=1}^n)$ for $\mathcal{X}$ with 
	\begin{align*}
		\|f_j\|=a_j,  \quad \|\tau_j\|=b_j,\quad |f_j(\tau_j)|=c_j, \quad \forall 1\leq j \leq n
	\end{align*}
if and only if 
	\begin{align*}
	&\max_{1\leq j \leq n}a_j^p\leq \frac{1}{d}\sum_{j=1}^{n}a_j^p,\\
	& \max_{1\leq j \leq n}b_j^q\leq \frac{1}{d}\sum_{j=1}^{n}b_j^q, \\
	& \max_{1\leq j \leq n}c_j^r\leq \frac{1}{d}\sum_{j=1}^{n}c_j^r
\end{align*}
for some real $p,q,r>0$}.
\end{conjecture}
A result in Hilbert space frame theory which is along with fundamental inequality is the following result.
\begin{theorem}\label{GSPECTRUM}\cite{CASAZZALEON, KORNELSON}
\textbf{Let 	 $n\geq d$ and $\mathcal{H}$ be a $d$-dimensional Hilbert   space and let $a_1, \dots, a_n$ be a collection of non negative real numbers. For any self adjoint positive operator $S:\mathcal{H}\to  \mathcal{H}$ with eigenvalues $\lambda_1\geq \lambda_2 \geq \cdots\geq  \lambda_d >0$, the following are equivalent. 
\begin{enumerate}[\upshape(i)]
	\item Then  exists a  frame  $\{\tau_j\}_{j=1}^{n}$  for $\mathcal{H}$ whose frame operator $S_\tau=S$ and  $\|\tau_j\|=a_j$, $\forall 1\leq j\leq n$. 
	\item For all $1\leq m\leq d$,
	\begin{align*}
		\sum_{k=1}^{m}a_k^2\leq \sum_{k=1}^{m}\lambda_k\quad \text{and} \quad 	\sum_{j=1}^{n}a_j^2= \sum_{k=1}^{d}\lambda_k.
	\end{align*}
\end{enumerate}}
\end{theorem}
Theorem \ref{GSPECTRUM} leads to the following conjecture. 
\begin{conjecture}\label{GCONJECTURE}
\textbf{Let 	 $n\geq d$ and $\mathcal{X}$ be a $d$-dimensional Banach    space and let $a_1, \dots, a_n,$ $  b_1, \dots, b_n$,  $  c_1\dots ,c_n$ be a collection of non negative real numbers. For any  operator $S:\mathcal{X}\to  \mathcal{X}$ with eigenvalues $\lambda_1\geq \lambda_2 \geq \cdots\geq  \lambda_d >0$, the following are equivalent. 
\begin{enumerate}[\upshape(i)]
	\item Then  exists an ASF  $(\{f_j\}_{j=1}^{n}, \{\tau_j\}_{j=1}^{n})$  for $\mathcal{X}$ whose frame operator $S_{f,\tau}=S$ and  $\|f_j\|=a_j$, $\|\tau_j\|=b_j$, $|f_j(\tau_j)|=c_j$, $\forall 1\leq j\leq n$. 
	\item For all $1\leq m\leq d$,
	\begin{align*}
		\sum_{k=1}^{m}a_k^p\leq \sum_{k=1}^{m}\lambda_k\quad \text{and} \quad 	\sum_{j=1}^{n}a_j^p= \sum_{k=1}^{d}\lambda_k,
	\end{align*}
	\begin{align*}
	\sum_{k=1}^{m}b_k^q\leq \sum_{k=1}^{m}\lambda_k\quad \text{and} \quad 	\sum_{j=1}^{n}b_j^q= \sum_{k=1}^{d}\lambda_k,
\end{align*}
	\begin{align*}
	\sum_{k=1}^{m}c_k^r\leq \sum_{k=1}^{m}\lambda_k\quad \text{and} \quad 	\sum_{j=1}^{n}c_j^r= \sum_{k=1}^{d}\lambda_k,
\end{align*}
for some real $p,q,r>0$.
\end{enumerate}}
\end{conjecture}
Celebrated scaling problem for frames for Hilbert spaces reads as follows and to state it, first we need a definition.
\begin{definition}\cite{KUTYNIOK}
A 	frame $\{\tau_n\}_n$ for a Hilbert space  $ \mathcal{H}$ is said to be \textbf{scalable} if there exists a sequence of (non negative) scalars $\{a_n\}_n$ such that 
\begin{align*}
	\{a_n\tau_n\}_n \text{  is a Parseval frame for } \mathcal{H}.
\end{align*}
\end{definition}
\begin{problem}\cite{KUTYNIOK, CASAZZABRIEF}
(\textbf{Scaling problem})	\textbf{Classify frames $\{\tau_n\}_n$ for a Hilbert space  $ \mathcal{H}$  so that there is a sequence of  scalars $\{a_n\}_n$ such that $\{a_n\tau_n\}_n$ is a Parseval frame   $ \mathcal{H}$, i.e., the frame $\{\tau_n\}_n$ is scalable.}
\end{problem}
Here we state scaling problem for Banach spaces with the introduction of scaling for p-approximate Schauder frames.
\begin{definition}
A 	p-ASF  $(\{f_n\}_n, \{\tau_n\}_n)$ for a Banach space  $ \mathcal{X}$ is said to be \textbf{p-scalable} if there exist  sequences of  scalars $\{a_n\}_n$, $\{b_n\}_n$ such that 
\begin{align*}
(\{a_nf_n\}_n, \{b_n\tau_n\}_n)	 \text{  is a Parseval p-ASF for } \mathcal{X}.
\end{align*}	
\end{definition}
\begin{problem}
(\textbf{Scaling problem for Banach spaces}) \textbf{Classify p-ASFs $(\{f_n\}_n, \{\tau_n\}_n)$ for a Banach  space  $ \mathcal{X}$  so that there are  sequences of  scalars $\{a_n\}_n$, $\{b_n\}_n$ such that $(\{a_nf_n\}_n, \{b_n\tau_n\}_n)$ is a Parseval p-ASF  for $ \mathcal{X}$, i.e., p-ASF $(\{f_n\}_n, \{\tau_n\}_n)$ is p-scalable.}
\end{problem}
Let us now recall the Kothe-Lorch theorem for Riesz basis for Hilbert spaces. 
\begin{theorem}\label{KOTHELORCH}\cite{HEIL, GOHBERG}
\textbf{(Kothe-Lorch theorem)}  For a collection 	$\{\tau_n\}_n$ in  a Hilbert space  $ \mathcal{H}$, the following are equivalent.
\begin{enumerate}[\upshape(i)]
	\item $\{\tau_n\}_n$ is a Riesz basis for  $ \mathcal{H}$.
	\item $\{\tau_n\}_n$ is an unconditional Schauder  basis for  $ \mathcal{H}$, and 
	\begin{align*}
		0<\inf_{n\in \mathbb{N}}\|\tau_n\|\leq \sup_{n\in \mathbb{N}}\|\tau_n\|<\infty.
	\end{align*}
\end{enumerate}
\end{theorem}
Based on Theorem \ref{KOTHELORCH}  we  formulate the following problem. 
\begin{problem}\label{KOTHELORCBANACH}\textbf{(Kothe-Lorch problem for Banach spaces)}
\textbf{For 	a collection 	$\{\tau_n\}_n$ in  a Banach  space  $ \mathcal{X}$ and a collection $\{f_n\}_n$ in    $ \mathcal{X}^*$, whether  the following are equivalent?
\begin{enumerate}[\upshape(i)]
	\item $(\{f_n\}_n, \{\tau_n\}_n)$ is a p-approximate Riesz basis for  $ \mathcal{X}$.
	\item $\{\tau_n\}_n$ is an unconditional Schauder  basis for  $ \mathcal{X}$, 
	\begin{align*}
		\phi(x)=\sum_{n=1}^{\infty}\phi(\tau_n)f_n(x), \quad \forall x \in \mathcal{X}, \forall \phi \in \mathcal{X}^*
	\end{align*}
and 
	\begin{align*}
		&0<\inf_{n\in \mathbb{N}}\|\tau_n\|\leq \sup_{n\in \mathbb{N}}\|\tau_n\|<\infty, \\
	& 0<\inf_{n\in \mathbb{N}}\|f_n\|\leq\sup_{n\in \mathbb{N}}\|f_n\|<\infty.
	\end{align*}
\end{enumerate}}
\end{problem}
Note that we always have (i) $\Rightarrow$ (ii) in Problem \ref{KOTHELORCBANACH}. 

For frames for Hilbert spaces there are three algorithms to get approximation to every element of the space. They are  Duffin-Schaeffer algorithm \cite{DUFFIN},   Grochenig-Chebyshev algorithm \cite{GROCHENIGACC} and Grochenig conjugate gradient algorithm \cite{GROCHENIGACC}. For p-ASFs we derive a Duffin-Schaeffer algorithm with stronger assumption and ask others as problems.
\begin{theorem}
	Let 	 $( \{f_n\}_{n},  \{\tau_n\}_{n} )$  be a    p-ASF for $\mathcal{X}$	 with bounds $a$ and  $b$.  For  $ x \in \mathcal{X}$, define 
$$ x_0\coloneqq0, ~x_n\coloneqq x_{n-1}+\frac{2}{a+b}S_{f,\tau}(x-x_{n-1}),\quad \forall n \geq1.$$
If $\|I_\mathcal{X} -\frac{2}{b+a}S_{f,\tau}\|\leq \frac{b-a}{b+a}$, 	then 
$$ \|x_n-x\|\leq \left(\frac{b-a}{b+a}\right)^n\|x\|, \quad\forall n \geq1.$$
In particular, $x_n\to x $ as $n \to \infty$.	
\end{theorem}
\begin{proof}
Using the definition of $x_n'$s, we see that 
\begin{align*}
	x-x_n&=x-x_{n-1}-\frac{2}{a+b}S_{f, \tau}(x-x_{n-1})\\
	&=\left(I_\mathcal{X} -\frac{2}{b+a}S_{f,\tau}\right)(x-x_{n-1})\\
	&=\cdots=\left(I_\mathcal{X} -\frac{2}{b+a}S_{f, \tau}\right)^nx, \quad\forall x \in \mathcal{X} ,  \forall n \geq 1.
\end{align*}	
Therefore
\begin{align*}
	\|x_n-x\|\leq \left\|I_\mathcal{X} -\frac{2}{b+a}S_{f,\tau}\right\|^n\|x\|\leq \left(\frac{b-a}{b+a}\right)^n\|x\|, \quad\forall n \geq1.
\end{align*}	
\end{proof}
\begin{problem} \textbf{(Frame algorithm problems for Banach spaces)}
\textbf{Whether there is a 
	\begin{enumerate}[\upshape(i)]
		\item Duffin-Schaeffer algorithm for p-ASFs?
		\item Grochenig-Chebyshev algorithm for p-ASFs?
		\item Grochenig conjugate gradient algorithm for p-ASFs?
	\end{enumerate}}
\end{problem}
In the groundbreaking paper \cite{ALDROUBICABRELLIMOLTERTANG}, Aldroubi, Cabrelli, Molter, and Tang introduced the notion of dynamical sampling for Hilbert spaces and solved for finite dimensional Hilbert spaces. The problem reads as follows. 
\begin{problem}\cite{ALDROUBICABRELLIMOLTERTANG, ALDROUBIHUANGPETROSYAN, ALDROUBICABRELLICAKMAKMOLTERPETROSYAN, DIAZROCIOMEDRIMOLTER}
	\label{DYNAMICALHILBERT}
\textbf{(Dynamical sampling  problem for Hilbert  spaces)} \textbf{Let $\mathcal{H}$ be a Hilbert space (finite or infinite dimensional).	Let $\mathbb{K}, \mathbb{M} \subseteq \mathbb{N}$,  $T:\mathcal{H}\to \mathcal{H}$ be a bounded linear operator and $\{\tau_k\}_{k\in \mathbb{K}}$ be a sequence in  $ \mathcal{H}$. Find conditions on $\mathbb{K}, \mathbb{M}$ and $T$ such that 
\begin{enumerate}[\upshape(i)]
	\item $\{T^m\tau_k\}_{k\in \mathbb{K}, m \in \mathbb{M}}$ is an orthonormal basis for   $ \mathcal{H}$.
	\item $\{T^m\tau_k\}_{k\in \mathbb{K}, m \in \mathbb{M}}$ is a Riesz  basis for   $ \mathcal{H}$.
		\item $\{T^m\tau_k\}_{k\in \mathbb{K}, m \in \mathbb{M}}$ is a Riesz  sequence  for   $ \mathcal{H}$.
	\item $\{T^m\tau_k\}_{k\in \mathbb{K}, m \in \mathbb{M}}$ is a frame  for   $ \mathcal{H}$.
		\item $\{T^m\tau_k\}_{k\in \mathbb{K}, m \in \mathbb{M}}$ is a Parseval frame  for   $ \mathcal{H}$.
	\item $\{T^m\tau_k\}_{k\in \mathbb{K}, m \in \mathbb{M}}$ is a Bessel sequence  for   $ \mathcal{H}$.
\end{enumerate}
If the operator $T$ is normal, then the set $\mathbb{M}$ can be replaced by any subset of $[0, \infty)$. }
\end{problem}

We formulate Problem  \ref{DYNAMICALHILBERT} for Banach spaces as follows.
\begin{problem}
\textbf{(Dynamical sampling  problem for Banach   spaces)} \textbf{Let $\mathcal{X}$ be a Banach  space (finite or infinite dimensional).	Let $\mathbb{K}, \mathbb{M} \subseteq \mathbb{N}$, $U:\mathcal{X}\to \mathcal{X}$, $V:\mathcal{X}\to \mathcal{X}$ be  bounded linear operators,  $\{\tau_k\}_{k\in \mathbb{K}}$ be a sequence in  $ \mathcal{X}$ and $\{f_k\}_{k\in \mathbb{K}}$ be a sequence in  $ \mathcal{X}^*$. Find conditions on $\mathbb{K}, \mathbb{M}$,  $U$ and $V$ such that 
	\begin{enumerate}[\upshape(i)]
		\item $(\{f_kU^m\}_{k\in \mathbb{K}, m \in \mathbb{M}}, \{V^m\tau_k\}_{k\in \mathbb{K}, m \in \mathbb{M}})$ is a p-orthonormal basis for   $ \mathcal{X}$.
		\item $(\{f_kU^m\}_{k\in \mathbb{K}, m \in \mathbb{M}}, \{V^m\tau_k\}_{k\in \mathbb{K}, m \in \mathbb{M}})$ is a p-approximate  Riesz  basis for   $ \mathcal{X}$.
		\item $(\{f_kU^m\}_{k\in \mathbb{K}, m \in \mathbb{M}}, \{V^m\tau_k\}_{k\in \mathbb{K}, m \in \mathbb{M}})$ is a p-approximate Riesz  sequence  for   $ \mathcal{X}$.
		\item $(\{f_kU^m\}_{k\in \mathbb{K}, m \in \mathbb{M}}, \{V^m\tau_k\}_{k\in \mathbb{K}, m \in \mathbb{M}})$ is a p-approximate Schauder frame  for   $ \mathcal{X}$.
			\item $(\{f_kU^m\}_{k\in \mathbb{K}, m \in \mathbb{M}}, \{V^m\tau_k\}_{k\in \mathbb{K}, m \in \mathbb{M}})$ is a Parseval p-approximate Schauder frame  for   $ \mathcal{X}$.
		\item $(\{f_kU^m\}_{k\in \mathbb{K}, m \in \mathbb{M}}, \{V^m\tau_k\}_{k\in \mathbb{K}, m \in \mathbb{M}})$ is a p-approximate Bessel sequence  for   $ \mathcal{X}$.
	\end{enumerate}
The set $\mathbb{M}$ can be replaced by subsets $\Sigma \subseteq \mathbb{C}$ whenever $U^\alpha $ and $V^\alpha$ make sense for $\alpha \in \Sigma$. For instance,  if $T$ is in the class of  bounded linear operators 	whose resolvent contains  $(-\infty,0]$   (sectorial operators) \cite{KOMATSU, MARKUSFUNCTIONAL}, then  we can take $\Sigma$ to be any subset of $\mathbb{C}$.}
\end{problem}
Now we recall the phase retrieval and norm retrieval in Hilbert spaces which originated from the fundamental work of Balan, Casazza, and Edidin \cite{BALANCASAZZAEDIDIN}.
\begin{definition}\cite{BALANCASAZZAEDIDIN, BAHMANPOURCAHILLCASAZZAJASPERWOODLAND, CAHILLCASAZZADAUBECHIES, BOTELHOANDRADECASAZZAVANTREMAIN}
	\begin{enumerate}[\upshape(i)]
		\item  \textbf{Let $\mathcal{H}$ be a Hilbert space (finite or infinite dimensional). A frame  $\{\tau_n\}_{n}$  for   $ \mathcal{H}$ does  phase retrieval   for $\mathcal{H}$ if $x,y  $  are any two elements of $\mathcal{H}$ satisfying 
			\begin{align*}
				|	\langle x, \tau_n \rangle |= |	\langle y, \tau_n \rangle |, \quad \forall n \in \mathbb{N},
			\end{align*}
			then $y=\alpha x$, for some $|\alpha|=1$.}
		\item  \textbf{Let $\mathcal{H}$ be a Hilbert space (finite or infinite dimensional). A frame  $\{\tau_n\}_{n}$  for   $ \mathcal{H}$ does  norm  retrieval   for $\mathcal{H}$ if $x,y  $  are any two elements of $\mathcal{H}$ satisfying 
			\begin{align*}
				|	\langle x, \tau_n \rangle |= |	\langle y, \tau_n \rangle |, \quad \forall n \in \mathbb{N},
			\end{align*}
			then $\|y\|=\|x\|$.}
	\end{enumerate}
\end{definition}
Following is the phase retrieval and norm retrieval problems for Hilbert spaces.
\begin{problem}\cite{BALANCASAZZAEDIDIN, BAHMANPOURCAHILLCASAZZAJASPERWOODLAND, CAHILLCASAZZADAUBECHIES, BOTELHOANDRADECASAZZAVANTREMAIN}
	\textbf{\begin{enumerate}[\upshape(i)]
		\item (Phase retrieval problem for Hilbert spaces) For a given Hilbert space $\mathcal{H}$, classify frames $\{\tau_n\}_{n}$  for $\mathcal{H}$ which does phase retrieval for  $\mathcal{H}$. 
		\item (Norm retrieval problem for Hilbert spaces) For a given Hilbert space for $\mathcal{H}$, classify frames $\{\tau_n\}_{n}$ for $\mathcal{H}$ which does  norm retrieval for $\mathcal{H}$.
	\end{enumerate}}
\end{problem}
Phase retrieval problem poses challenges even in finite dimensions. For instance, phase retrieval can be classified only for real Hilbert spaces using complement property \cite{BANDEIRACAHILLDUSTINNELSON}. Number of vectors needed and a full description of them to do phase retrieval in arbitrary dimension is still open \cite{CASAZZAWOODLANDPHASE, CONCAEDIDHERINGVINZANT, VINZANT, BODMANNHAMMEN, FICKUSMIXONNELSONWANG}.\\
We now state the phase retrieval and norm retrieval notions and problems for p-approximate Schauder frames for Banach space. 
\begin{definition}
	\begin{enumerate}[\upshape(i)]
	\item  \textbf{Let $\mathcal{X}$ be a Banach  space (finite or infinite dimensional). A p-approximate Schauder frame   $(\{f_n\}_{n}, \{\tau_n\}_{n})$  for   $ \mathcal{X}$ does  phase retrieval   for $\mathcal{X}$ if $x,y  $  are any two elements of $\mathcal{X}$ and $\phi,\psi  $  are any two elements of $\mathcal{X}^*$   satisfying 
		\begin{align*}
		|\phi(\tau_n)|=|\psi(\tau_n)	|	\quad \text{and}	\quad |f_n(x)|=|f_n(y)|, \quad \forall n \in \mathbb{N},
		\end{align*}
		then $y=\alpha x$, for some $|\alpha|=1$ and  $\psi=\beta \phi$, for some $|\beta|=1$}.
	\item  \textbf{Let $\mathcal{X}$ be a Banach  space (finite or infinite dimensional). A p-approximate Schauder frame $(\{f_n\}_{n}, \{\tau_n\}_{n})$  for   $ \mathcal{X}$ does  norm  retrieval   for $\mathcal{X}$ if $x,y  $  are any two elements of $\mathcal{X}$ and $\phi,\psi  $  are any two elements of $\mathcal{X}^*$ satisfying 
		\begin{align*}
	|\phi(\tau_n)|=|\psi(\tau_n)	|	\quad \text{and}	\quad	|f_n(x)|=|f_n(y)|, \quad \forall n \in \mathbb{N},
		\end{align*}
		then $\|y\|=\|x\|$ and  $\|\phi\|=\|\psi\|$.}
\end{enumerate}	
\end{definition}

\begin{problem}
	\textbf{\begin{enumerate}[\upshape(i)]
		\item (Phase retrieval problem for Banach  spaces) For a given Banach  space    $ \mathcal{X}$, classify p-approximate Schauder frames  $(\{f_n\}_{n}, \{\tau_n\}_{n})$ for $\mathcal{X}$    which does phase retrieval for   $ \mathcal{X}$. 
		\item (Norm retrieval problem for Banach  spaces) For a given Banach  space    $ \mathcal{X}$, classify p-approximate Schauder frames  $(\{f_n\}_{n}, \{\tau_n\}_{n})$ for   $ \mathcal{X}$ which does  norm retrieval for   $ \mathcal{X}$.
\end{enumerate}}	
\end{problem}
Our next problem is motivated from the works of Dykema and Strawn in Hilbert spaces \cite{CAHILLMIXONSTRAWN, STRAWNBOOK, NEEDHAMSHONKWILER, STRAWN2011, DYKEMASTRAWN, BODMANNHAAS2015, GIOL, STRAWNOPTIMIZATION} and Corach, Pacheco, and Stojanoff \cite{CORACHPACHECOSTOJANOFF}.
\begin{problem} 	\textbf{(Algebraic geometry and topological problems for approximate Schauder frames for Banach spaces)}
	\textbf{Let $n\geq d$ be fixed natural numbers and $\mathcal{X}$ be a fixed $d$-dimensional Banach space. Define 
	\begin{align*}
		\mathscr{F}^n(\mathbb{K}^d)\coloneqq \{(&\{f_j\}_{j=1}^n, \{\tau_j\}_{j=1}^n): (\{f_j\}_{j=1}^n, \{\tau_j\}_{j=1}^n) \text{ is a tight ASF for } \mathcal{X} \text{ such that}\\
		& \|f_j\|=\|\tau_j\|=|f_j(\tau_j)|=1, \forall 1\leq j \leq n\}.
	\end{align*}
\begin{enumerate}[\upshape(i)]
	\item Whether there is an explicit description of $	\mathscr{F}^n(\mathbb{K}^d)$?
	\item What is a  (good) topology on $	\mathscr{F}^n(\mathbb{K}^d)$? In particular, whether $	\mathscr{F}^n(\mathbb{K}^d)$ is path-connected?
	\item Whether $	\mathscr{F}^n(\mathbb{K}^d)$ is a manifold?
\end{enumerate}}
\end{problem}
\begin{problem} 	\textbf{(Topological problem for p-approximate Schauder frames for Banach spaces)}
	\textbf{Let $p\in[1, \infty), $ and $\mathcal{X}$ be a fixed Banach space (finite or infinite dimensional). Define 
		\begin{align*}
			\mathscr{F}^p(\mathcal{X})\coloneqq \{(\{f_n\}_{n}, \{\tau_n\}_{n}&: (\{f_n\}_{n}, \{\tau_n\}_{n}) \text{ is a p-ASF for } \mathcal{X}\}.
		\end{align*}
		\begin{enumerate}[\upshape(i)]
			\item Whether there is an explicit description of $	\mathscr{F}^p(\mathcal{X})$?
			\item What is a  (good) topology on $	\mathscr{F}^p(\mathcal{X})$? In particular, whether $		\mathscr{F}^p(\mathcal{X})$ is path-connected?
			\item Whether $		\mathscr{F}^p(\mathcal{X})$ is a Banach manifold?
	\end{enumerate}}
\end{problem}

We observe that there are two important results in the theory of frames for Hilbert spaces, one connects Bessel sequences to frames and another connects frames to Riesz bases. First one,  due to Li and Sun \cite{LISUN} states that  every Bessel sequence can be extended to a frame by adding extra elements, if necessary. Second one, known as Naimark-Han-Larson dilation theorem \cite{HANLARSON2000, CZAJA, KASHINKULIKOVA} states that a  frame for a  Hilbert space can be dilated to a Riesz basis.    For p-approximate Schauder frames (even for approximate Schauder frames), using a result (which used approximation properties \cite{CASAZZA2001})  of Casazza and Christensen \cite{CASAZZA2008} it is proved by Krishna and Johnson \cite{KRISHNAJOHNSON3} that p-approximate Bessel sequences may not be always  extended to p-approximate Schauder frames. However, using explicit construction, Krishna and Johnson  \cite{MAHESHJOHNSON2} proved that p-approximate Schauder frame can always be dilated to a p-approximate Riesz basis. It is also proved in  \cite{MAHESHJOHNSON2} that the dilation is optimal, i.e., it will not add anything extra if the given collection is already p-approximate Riesz basis.

We next introduce the notion of continuous p-approximate Schauder frames. Our motivation is from the notion of continuous frames for Hilbert spaces introduced independently by Ali, Antoine, and Gazeau, \cite{ALIANTOINEGAZEAU} and Kaiser \cite{KAISERBOOK}, (see \cite{GABARDOHANFRAMESASSOCITED2003, RAHIMI}) and from the notion of continuous Schauder frames by Eisner and Freeman \cite{EISNERFREEMAN} (also see \cite{LILIHAN}).  All  integrals which appear in sequel are in weak sense (Pettis integrals \cite{TALAGRAND}).
\begin{definition}
Let $(\Omega, \mu)	$ be a measure space  and $p\in[1, \infty)$. Let $\{\tau_\alpha\}_{\alpha\in \Omega}$ be a collection  in a Banach space  $\mathcal{X}$ and 	$\{f_\alpha\}_{\alpha\in \Omega}$ be a collection in  $\mathcal{X}^*$ (dual of  $\mathcal{X}$). The pair $ (\{f_\alpha \}_{\alpha \in \Omega}, \{\tau_\alpha \}_{\alpha \in \Omega}) $ is said to be a \textbf{continuous p-approximate Schauder frame (continuous p-ASF)} for $\mathcal{X}$ if the following conditions are satisfied.
\begin{enumerate}[\upshape(i)]
	\item For each $x \in \mathcal{X}$, the map $\Omega \ni \alpha \mapsto  f_\alpha(x)\in \mathbb{R} \text{ or } \mathbb{C}$ is measurable. 
	\item For each $u \in \mathcal{L}^p(\Omega, \mu)$, the map $\Omega \ni \alpha \mapsto  u(\alpha)\tau_\alpha \in \mathcal{X}$ is measurable. 
	\item The map  (\textbf{continuous analysis operator})
	\begin{align*}
		\theta_f: \mathcal{X}\ni x \mapsto \theta_f  \in \mathcal{L}^p(\Omega, \mu); \quad \theta_f x: \Omega \ni \alpha \mapsto (\theta_f x)(\alpha)\coloneqq f_\alpha(x)\in \mathbb{R} \text{ or } \mathbb{C}
	\end{align*}
	is a well-defined bounded linear  operator.
	\item The map  (\textbf{continuous synthesis operator})
	\begin{align*}
		\theta_\tau : \mathcal{L}^p(\Omega, \mu)\ni u \mapsto \theta_\tau u\coloneqq \int_\Omega u(\alpha)\tau_\alpha \,d\mu(\alpha)\in \mathcal{X}	
	\end{align*}
	is a well-defined bounded linear  operator.
		\item The map (\textbf{continuous frame operator})
	\begin{align*}
		S_{f, \tau}:\mathcal{X}\ni x \mapsto S_{f, \tau}x\coloneqq \int_\Omega f_\alpha(x)\tau_\alpha \,d\mu(\alpha) \in
		\mathcal{X}
	\end{align*}
	is a well-defined bounded linear, invertible operator.	
\end{enumerate}
Constants $a>0$ and $b>0$ satisfying 
\begin{align*}
	a\|x\|\leq \|S_{f,\tau}	x\|\leq b \|x\|, \quad \forall x \in \mathcal{X},
\end{align*}
are called as \textbf{lower frame bound} and \textbf{upper frame bound}, respectively. If $S_{f,\tau}x=x$, $\forall x \in \mathcal{X}$, then we say that $ (\{f_\alpha \}_{\alpha \in \Omega}, \{\tau_\alpha \}_{\alpha \in \Omega}) $ is a Parseval continuous p-ASF for $\mathcal{X}$.  If $S_{f,\tau}x=\lambda x$, for some nonzero scalar $\lambda$, $\forall x \in \mathcal{X}$, then we say that $ (\{f_\alpha \}_{\alpha \in \Omega}, \{\tau_\alpha \}_{\alpha \in \Omega}) $ is a tight continuous p-ASF for $\mathcal{X}$. If we do not demand  condition (v), then we say that $ (\{f_\alpha \}_{\alpha \in \Omega}, \{\tau_\alpha \}_{\alpha \in \Omega}) $ is a \textbf{continuous p-approximate Bessel sequence (p-ABS)} for $\mathcal{X}$. Constant  $b$ is called as \textbf{Bessel bound}.	
\end{definition}
Here is a genuine example.
\begin{example}
Let $p\in (1, \infty)$ and $q$ be the conjugate index of $p$. 	Define $\mathcal{X}\coloneqq \mathbb{R}^2$ equipped with $\|\cdot\|_1$ norm. Define $\Omega\coloneqq[0,2\pi]$ equipped with the usual Lebesgue measure. Now define 
	\begin{align*}
		\tau_\alpha\coloneqq (\cos \alpha, \sin\alpha), \quad \forall \alpha \in [0,2\pi].
	\end{align*}
and 

\begin{align*}
	f_\alpha: \mathcal{X} \ni (x,y)\mapsto f_\alpha(x,y)\coloneqq x \cos \alpha+y \sin \alpha\in \mathbb{R}, \quad \forall \alpha \in [0,2\pi].
\end{align*}
Then 
\begin{align*}
	\int_\Omega f_\alpha(x,y) \tau_\alpha \,d\mu(\alpha)&=\int_{0}^{2\pi}(x \cos \alpha+y \sin \alpha ) (\cos \alpha, \sin\alpha)\,d\alpha\\
	&=\left(\int_{0}^{2\pi}(x \cos \alpha+y \sin \alpha ) \cos \alpha \,d\alpha, \int_{0}^{2\pi}(x \cos \alpha+y \sin \alpha ) \sin\alpha\,d\alpha\right)\\
	&=\pi(x,y), \quad \forall (x,y)\in \mathbb{R}^2,
\end{align*}

\begin{align*}
	\|\theta_f(x,y)\|&=\left(\int_{0}^{2\pi}|(\theta_f(x,y))|^p\,d\alpha\right)^\frac{1}{p}=\left(\int_{0}^{2\pi}|f_\alpha(x,y))|^p\,d\alpha\right)^\frac{1}{p}\\
	&=\left(\int_{0}^{2\pi}|x \cos \alpha+y \sin \alpha|^p\,d\alpha\right)^\frac{1}{p}\\
	&\leq \left(\int_{0}^{2\pi}|x \cos \alpha|^p\,d\alpha\right)^\frac{1}{p}+\left(\int_{0}^{2\pi}|y \sin \alpha|^p\,d\alpha\right)^\frac{1}{p}\\
	&= |x|\left(\int_{0}^{2\pi}| \cos \alpha|^p\,d\alpha\right)^\frac{1}{p}+|y|\left(\int_{0}^{2\pi}| \sin \alpha|^p\,d\alpha\right)^\frac{1}{p}\\
	&=(2\pi)^\frac{1}{p}(|x|+|y|)=(2\pi)^\frac{1}{p}\|(x,y)\|_1, \quad \forall (x,y) \in \mathbb{R}^2,
\end{align*}
and 
\begin{align*}
	\|\theta_\tau u\|_1&=\left\|\int_{0}^{2\pi}u(\alpha)\tau_\alpha\,d\alpha\right\|_1=\left\|\int_{0}^{2\pi}u(\alpha)(\cos \alpha, \sin\alpha)\,d\alpha\right\|_1\\
	&=\left\|\left(\int_{0}^{2\pi}u(\alpha)\cos \alpha\,d\alpha, \int_{0}^{2\pi}u(\alpha) \sin\alpha\,d\alpha\right)\right\|_1\\
	&=\left|\int_{0}^{2\pi}u(\alpha)\cos \alpha\,d\alpha\right|+\left|\int_{0}^{2\pi}u(\alpha)\sin \alpha\,d\alpha\right|\\
	&\leq \left(\int_{0}^{2\pi}|u(\alpha)|^p\,d\alpha\right)^\frac{1}{p}\left(\int_{0}^{2\pi}|\cos \alpha|^q\,d\alpha\right)^\frac{1}{q}+\left(\int_{0}^{2\pi}|u(\alpha)|^p\,d\alpha\right)^\frac{1}{p}\left(\int_{0}^{2\pi}|\sin \alpha|^q\,d\alpha\right)^\frac{1}{q}\\
	&\leq 2(2\pi)^\frac{1}{q} \left(\int_{0}^{2\pi}|u(\alpha)|^p\,d\alpha\right)^\frac{1}{p}=2(2\pi)^\frac{1}{q}\|u\|_p, \quad \forall u \in \mathcal{L}^p[0,2\pi].
\end{align*}
Therefore  $ (\{f_\alpha \}_{\alpha \in \Omega}, \{\tau_\alpha \}_{\alpha \in \Omega}) $  is a continuous p-ASF for $\mathbb{R}^2$. 
Note that  $ (\{f_\alpha \}_{\alpha \in \Omega}, \{\tau_\alpha \}_{\alpha \in \Omega}) $ is also continuous 1-ASF.
\end{example}
Here is discretization problem for continous p-approximate Schauder frames  for Banach spaces. For frames for Hilbert spaces this was asked by  Ali, Antoine, and Gazeau \cite{ALIANTOINEGAZEAUBOOK} and solved by Fornasier, Rauhut, Freeman,  and Speegle \cite{FREEMAN2019, FORNASIERRAUHUT, RAUHUTULLRICH2011}. The problem for continuous Schauder frames has been asked by Eisner and Freeman is still open  \cite{EISNERFREEMAN}.
\begin{problem}\textbf{(Discretization problem for continuous p-approximate Schauder frames for Banach spaces)}
	\textbf{When a continuous p-approximate Schauder frame for a Banach space can be sampled to obtain a (discrete) p-approximate Schauder frame?}
\end{problem}
We now introduce the notion of continuous p-approximate Riesz basis for Banach spaces. Motivation comes from Theorem \ref{RIESZCHAR}.
\begin{definition}
A continuous p-ASF $ (\{f_\alpha \}_{\alpha \in \Omega}, \{\tau_\alpha \}_{\alpha \in \Omega}) $ for $\mathcal{X}$	is said to be \textbf{continuous p-approximate Riesz basis} if 
\begin{align*}
		\theta_fS_{f,\tau}^{-1}\theta_\tau=I_{\mathcal{L}^p(\Omega, \mu)}.	
\end{align*}
\end{definition}
\begin{definition}
 Let $\{\tau_\alpha\}_{\alpha\in \Omega}$ be a collection  in a Banach space  $\mathcal{X}$ and 	$\{f_\alpha\}_{\alpha\in \Omega}$ be a collection in  $\mathcal{X}^*$ (dual of  $\mathcal{X}$). The pair $ (\{f_\alpha \}_{\alpha \in \Omega}, \{\tau_\alpha \}_{\alpha \in \Omega}) $ is said to be a \textbf{continuous p-approximate Riesz family  } for $\mathcal{X}$ if $ (\{f_\alpha \}_{\alpha \in \Omega}, \{\tau_\alpha \}_{\alpha \in \Omega}) $	is a continuous  p-approximate Riesz basis    for $\overline{span}\{\tau_\alpha\}_{\alpha \in \Omega}$.
\end{definition}

Now we can formulate most of the notions,  conjectures and problems stated earlier for continuous p-approximate Schauder frames. Here are  some samples. Throughout,  $(\Omega, \mu)	$ is a measure space and $ (\{f_\alpha \}_{\alpha \in \Omega}, \{\tau_\alpha \}_{\alpha \in \Omega}) $ is at least  continuous p-approximate Bessel family whenever integrals appear.
\begin{conjecture}
\textbf{(Feichtinger conjecture for continuous p-approximate Schauder frames)
	Let  $ (\{f_\alpha \}_{\alpha \in \Omega}, \{\tau_\alpha \}_{\alpha \in \Omega}) $  be  a continuous p-ASF for   $\mathcal{X}$ such that 
	\begin{align*}
		&0<\inf_{\alpha \in \Omega}\|\tau_\alpha\|\leq \sup_{\alpha \in \Omega}\|\tau_\alpha\|<\infty ,\\
		&0<\inf_{\alpha \in \Omega}\|f_\alpha\|\leq \sup_{\alpha \in \Omega}\|f_\alpha\|<\infty.
	\end{align*}
	Then $ (\{f_\alpha \}_{\alpha \in \Omega}, \{\tau_\alpha \}_{\alpha \in \Omega}) $  can be partitioned into a finite union of continuous p-approximate Riesz families. Moreover, what is the number of partitions required?}	
\end{conjecture}
\begin{conjecture}
	(\textbf{Continuous Weaver's conjecture for Banach spaces})
	\textbf{Let  $\mathcal{X}$ be a $d$-dimensional Banach  space.	There are universal constants $b\geq 2$ and $ b>\varepsilon > 0$ such that the following holds. Let $ \{\tau_\alpha \}_{\alpha \in \Omega}$ be a family in  $\mathcal{X}$ and $ \{f_\alpha \}_{\alpha \in \Omega}$ be a family in  $\mathcal{X}^*$ satisfying: 
		\begin{align*}
		\|f_\alpha\|\leq 1, ~ \|\tau_\alpha\|\leq 1, \quad \forall \alpha \in \Omega
		\end{align*}
		and 
		\begin{align*}
			\left\|\int_{\Omega} f_\alpha(x)\tau_\alpha \, d\mu(\alpha)\right\|\leq b\|x\|, \quad \forall x \in \mathcal{X}.
		\end{align*}
		Then there exists a measurable partition $\Delta_1, \dots, \Delta_M$ of $\Omega$ such that 
		\begin{align*}
			\left\|\int_{\Delta_k} f_\alpha(x)\tau_\alpha \, d\mu(\alpha)\right\|\leq (b-\varepsilon)\|x\|, \quad \forall x \in \mathcal{X}, \forall 1\leq k \leq M.
	\end{align*} }
\end{conjecture}
\begin{conjecture}
	(\textbf{Continuous Weaver's conjecture for Banach spaces - strong form})
	\textbf{Let  $\mathcal{X}$ be a $d$-dimensional Banach  space.	There are universal constants $b\geq 2$ and $ b>\varepsilon > 0$ such that the following holds. Let $ \{\tau_\alpha \}_{\alpha \in \Omega}$ be a family in  $\mathcal{X}$ and  $ \{f_\alpha \}_{\alpha \in \Omega}$ be a family in  $\mathcal{X}^*$ satisfying: 
		\begin{align*}
			\|f_\alpha\|\leq 1, ~ \|\tau_\alpha\|\leq 1, \quad \forall \alpha \in \Omega, 
		\end{align*}
		the spectrum of the  operator $S_{f,\tau}:\mathcal{X} \ni x \mapsto \int_{\Omega} f_\alpha(x)\tau_\alpha \, d\mu(\alpha) \in \mathcal{X}$  lies in $[0,\infty)$ 
		and 
		\begin{align*}
			\left\|\int_{\Omega} f_\alpha(x)\tau_\alpha \, d\mu(\alpha)\right\|\leq b\|x\|, \quad \forall x \in \mathcal{X}.
		\end{align*}
		Then there exists a measurable partition $\Delta_1, \dots, \Delta_M$ of $\Omega$ such that 
		\begin{align*}
			\left\|\int_{\Delta_k} f_\alpha(x)\tau_\alpha \, d\mu(\alpha)\right\|\leq (b-\varepsilon)\|x\|, \quad \forall x \in \mathcal{X}, \forall 1\leq k \leq M.
	\end{align*} }	
\end{conjecture}
\begin{conjecture}
	(\textbf{Continuous Akemann-Weaver conjecture for Banach spaces}) \textbf{ Let  $\mathcal{X}$ be a  Banach  space (finite or infinite dimensional). There exists a universal constant $c$  such that the following holds. 	Let 	$ (\{f_\alpha \}_{\alpha \in \Omega}, \{\tau_\alpha \}_{\alpha \in \Omega}) $ be a continuous p-approximate Bessel family  with Bessel bound 1   for $\mathcal{X}$ satisfying 
		\begin{align*}
			\sup_{\alpha \in \Omega}\|f_\alpha\|^q\leq \varepsilon, \quad 	\sup_{\alpha \in \Omega}\|\tau_\alpha\|^s\leq \varepsilon,
		\end{align*}
		for some $\epsilon>0$, for some $q,s>0$. Let $\{r_\alpha \}_{\alpha \in \Omega}$ be any family in [0,1]. Then there exists a measurable subset $\Delta \subseteq \Omega$ and a $d>0$  such that 
		\begin{align*}
			\left\| \int_{\Delta} f_\alpha(x)\tau_\alpha \, d\mu(\alpha) -\int_{\Omega } r_\alpha f_\alpha(x)\tau_\alpha \, d\mu(\alpha) \right\|\leq c \varepsilon^\frac{1}{d}\|x\|, \quad \forall x \in \mathcal{X},
		\end{align*}
		i.e., 
		\begin{align*}
			\left\| \int_{\Delta} f_\alpha(\cdot)\tau_\alpha \, d\mu(\alpha) -\int_{\Omega } r_\alpha f_\alpha(\cdot)\tau_\alpha \, d\mu(\alpha) \right\|\leq c \varepsilon^\frac{1}{d}.
	\end{align*}}	
\end{conjecture}
\begin{conjecture}
(\textbf{Continuous Akemann-Weaver conjecture for Banach spaces - strong form})		\textbf{ Let  $\mathcal{X}$ be a Banach  space (finite or infinite dimensional). There exists a universal constant $c$  such that the following holds. 	Let 	$ (\{f_\alpha \}_{\alpha \in \Omega}, \{\tau_\alpha \}_{\alpha \in \Omega}) $ be a continuous p-approximate Bessel family  with Bessel bound 1   for $\mathcal{X}$ satisfying 
	\begin{align*}
		\sup_{\alpha \in \Omega}\|f_\alpha\|^q\leq \varepsilon, \quad 	\sup_{\alpha \in \Omega}\|\tau_\alpha\|^s\leq \varepsilon,
	\end{align*}
	for some $\epsilon>0$, for some $q,s>0$ and the spectrum of the  operator $S_{f,\tau}:\mathcal{X} \ni x \mapsto \int_{\Omega } f_\alpha(x)\tau_\alpha \, d\mu(\alpha) \in \mathcal{X}$  lies in $[0,\infty)$. Let $\{r_\alpha \}_{\alpha \in \Omega}$ be any family in [0,1]. Then there exists a measurable subset $\Delta \subseteq \Omega$ and a $d>0$  such that 
	\begin{align*}
		\left\| \int_{\Delta} f_\alpha(x)\tau_\alpha \, d\mu(\alpha) -\int_{\Omega } r_\alpha f_\alpha(x)\tau_\alpha \, d\mu(\alpha) \right\|\leq c \varepsilon^\frac{1}{d}\|x\|, \quad \forall x \in \mathcal{X},
	\end{align*}
	i.e., 
	\begin{align*}
		\left\| \int_{\Delta} f_\alpha(\cdot)\tau_\alpha \, d\mu(\alpha) -\int_{\Omega } r_\alpha f_\alpha(\cdot)\tau_\alpha \, d\mu(\alpha) \right\|\leq c \varepsilon^\frac{1}{d}.
\end{align*}}	
\end{conjecture}
\begin{definition}
	A 	continuous p-ASF $ (\{f_\alpha \}_{\alpha \in \Omega}, \{\tau_\alpha \}_{\alpha \in \Omega}) $ for a Banach space  $ \mathcal{X}$ is said to be \textbf{p-scalable} if there exist  families  of  scalars $\{a_\alpha\}_{\alpha\in \Omega}$, $\{b_\alpha\}_{\alpha \in \Omega}$ such that 
	\begin{align*}
	 (\{a_\alpha f_\alpha \}_{\alpha \in \Omega}, \{b_\alpha \tau_\alpha \}_{\alpha \in \Omega}) 	 \text{  is a Parseval continuous p-ASF for } \mathcal{X}.
	\end{align*}	
\end{definition}
\begin{problem}
	(\textbf{Continuous scaling problem for Banach spaces}) \textbf{Classify continuous p-ASFs $ (\{f_\alpha \}_{\alpha \in \Omega},  \{\tau_\alpha \}_{\alpha \in \Omega}) $ for a Banach  space  $ \mathcal{X}$  so that there are  families  of  scalars $\{a_\alpha\}_{\alpha\in \Omega}$, $\{b_\alpha\}_{\alpha \in \Omega}$ such that $(\{a_\alpha f_\alpha \}_{\alpha \in \Omega}, \{b_\alpha \tau_\alpha \}_{\alpha \in \Omega}) $ is a Parseval continuous p-ASF  for $ \mathcal{X}$, i.e., continuous p-ASF $(\{ f_\alpha \}_{\alpha \in \Omega}, \{\tau_\alpha \}_{\alpha \in \Omega}) $ is p-scalable.}
\end{problem}
\begin{problem} \textbf{(Continuous frame algorithm problems for Banach spaces)}
	\textbf{Whether there is a 
		\begin{enumerate}[\upshape(i)]
			\item Duffin-Schaeffer algorithm for continuous p-ASFs?
			\item Grochenig-Chebyshev algorithm for continuous p-ASFs?
			\item Grochenig conjugate gradient algorithm for continuous p-ASFs?
	\end{enumerate}}
\end{problem}
\begin{problem}
	\textbf{(Continuous dynamical sampling  problem for Banach   spaces)} \textbf{Let $\mathcal{X}$ be a Banach  space (finite or infinite dimensional).	Let $\mathbb{M} \subseteq \mathbb{N}$ and $\Delta$ be a measurable subset of $\Omega$, $U:\mathcal{X}\to \mathcal{X}$, $V:\mathcal{X}\to \mathcal{X}$ be  bounded linear operators,  $\{\tau_\alpha\}_{\alpha \in \Delta}$ be a family in  $ \mathcal{X}$ and $\{f_\alpha\}_{\alpha \in \Delta}$ be a family  in  $ \mathcal{X}^*$. Find conditions on $\Delta, \mathbb{M}$,  $U$ and $V$ such that 
		\begin{enumerate}[\upshape(i)]
			\item $(\{f_\alpha U^m\}_{\alpha \in \Delta, m \in \mathbb{M}}, \{V^m\tau_\alpha \}_{\alpha \in \Delta, m \in \mathbb{M}})$ is a continuous p-approximate  Riesz  basis for   $ \mathcal{X}$.
			\item $(\{f_\alpha U^m\}_{\alpha \in \Delta, m \in \mathbb{M}}, \{V^m\tau_\alpha \}_{\alpha \in \Delta, m \in \mathbb{M}})$ is a continuous p-approximate Riesz  family  for   $ \mathcal{X}$.
			\item $(\{f_\alpha U^m\}_{\alpha \in \Delta, m \in \mathbb{M}}, \{V^m\tau_\alpha \}_{\alpha \in \Delta, m \in \mathbb{M}})$ is a continuous p-approximate Schauder frame  for   $ \mathcal{X}$.
			\item $(\{f_\alpha U^m\}_{\alpha \in \Delta, m \in \mathbb{M}}, \{V^m\tau_\alpha \}_{\alpha \in \Delta, m \in \mathbb{M}})$ is a Parseval continuous p-approximate Schauder frame  for   $ \mathcal{X}$.
			\item $(\{f_\alpha U^m\}_{\alpha \in \Delta, m \in \mathbb{M}}, \{V^m\tau_\alpha \}_{\alpha \in \Delta, m \in \mathbb{M}})$ is a continuous p-approximate Bessel family  for   $ \mathcal{X}$.
		\end{enumerate}
		The set $\mathbb{M}$ can be replaced by subsets $\Sigma \subseteq \mathbb{C}$ whenever $U^\alpha $ and $V^\alpha$ make sense for $\alpha \in \Sigma$.}
\end{problem}
\begin{definition}
	\begin{enumerate}[\upshape(i)]
		\item  \textbf{Let $\mathcal{X}$ be a Banach  space (finite or infinite dimensional). A continuous  p-approximate Schauder frame   $ (\{f_\alpha \}_{\alpha \in \Omega}, \{\tau_\alpha \}_{\alpha \in \Omega}) $ for   $ \mathcal{X}$ does  continous phase retrieval   for $\mathcal{X}$ if $x,y  $  are any two elements of $\mathcal{X}$ and $\phi,\psi  $  are any two elements of $\mathcal{X}^*$   satisfying 
			\begin{align*}
				|\phi(\tau_\alpha)|=|\psi(\tau_\alpha)|	\quad \text{and}	\quad |f_\alpha(x)|=|f_\alpha(y)|, \quad \forall \alpha \in \Omega,
			\end{align*}
			then $y=\alpha x$, for some $|\alpha|=1$ and  $\psi=\beta \phi$, for some $|\beta|=1$}.
		\item  \textbf{Let $\mathcal{X}$ be a Banach  space (finite or infinite dimensional). A continuous  p-approximate Schauder frame $ (\{f_\alpha \}_{\alpha \in \Omega}, \{\tau_\alpha \}_{\alpha \in \Omega}) $  for   $ \mathcal{X}$ does  continuous norm  retrieval   for $\mathcal{X}$ if $x,y  $  are any two elements of $\mathcal{X}$ and $\phi,\psi  $  are any two elements of $\mathcal{X}^*$ satisfying 
			\begin{align*}
				|\phi(\tau_\alpha)|=|\psi(\tau_\alpha)	|	\quad \text{and}	\quad	|f_\alpha(x)|=|f_\alpha(y)|, \quad \forall \alpha \in \Omega,
			\end{align*}
			then $\|y\|=\|x\|$ and  $\|\phi\|=\|\psi\|$.}
	\end{enumerate}	
\end{definition}

\begin{problem}
	\textbf{\begin{enumerate}[\upshape(i)]
			\item (Continuous phase retrieval problem for Banach  spaces) For a given Banach  space    $ \mathcal{X}$, classify continuous p-approximate Schauder frames  $ (\{f_\alpha \}_{\alpha \in \Omega}, \{\tau_\alpha \}_{\alpha \in \Omega}) $ for $\mathcal{X}$    which does continuous phase retrieval for   $ \mathcal{X}$. 
			\item (Continuous norm retrieval problem for Banach  spaces) For a given Banach  space    $ \mathcal{X}$, classify continous p-approximate Schauder frames  $ (\{f_\alpha \}_{\alpha \in \Omega}, \{\tau_\alpha \}_{\alpha \in \Omega}) $ for   $ \mathcal{X}$ which does  continuous norm retrieval for   $ \mathcal{X}$.
	\end{enumerate}}	
\end{problem}

 We end by asking  the following interesting question. 

\begin{problem}
	\textbf{If any of the  Conjectures \ref{FB},  \ref{FS}, \ref{3},   \ref{4},      \ref{11},   \ref{12},  \ref{121},   \ref{122}, \ref{123}, \ref{1231},  \ref{1232}, \ref{1233}, \ref{1234}, \ref{AW},  \ref{AW2}, \ref{13}, \ref{14}, \ref{141}, \ref{FUNDAMENTALINEQUALITY}, \ref{GCONJECTURE}    fails in a given dimension, classify  the spaces and/or p-approximate Schauder frames and/or p-approximate Bessel sequences in which the corresponding conjecture holds.}
\end{problem}

 \bibliographystyle{plain}
 \bibliography{reference.bib}

\end{document}